\documentclass{amsart}
\usepackage{amsmath}
\usepackage{amssymb}
\usepackage{amsthm}

\newcommand{\theoremref}[1]{\hyperref[#1]{Theorem~\ref*{#1}}}
\newcommand{\claimref}[1]{\hyperref[#1]{Claim~\ref*{#1}}}
\newcommand{\situationref}[1]{\hyperref[#1]{Situation~\ref*{#1}}}
\newcommand{\lemmaref}[1]{\hyperref[#1]{Lemma~\ref*{#1}}}
\newcommand{\definitionref}[1]{\hyperref[#1]{Definition~\ref*{#1}}}
\newcommand{\propositionref}[1]{\hyperref[#1]{Proposition~\ref*{#1}}}
\newcommand{\conjectureref}[1]{\hyperref[#1]{Conjecture~\ref*{#1}}}
\newcommand{\corollaryref}[1]{\hyperref[#1]{Corollary~\ref*{#1}}}
\newcommand{\exerciseref}[1]{\hyperref[#1]{Exercise~\ref*{#1}}}

\usepackage[mathscr]{euscript}

\usepackage[all]{xy}

\xyoption{2cell}
\UseAllTwocells

\usepackage{multicol}

\usepackage{xr-hyper}

\usepackage{hyperref}
\usepackage[usenames,dvipsnames]{xcolor}
\hypersetup{colorlinks=true,citecolor=NavyBlue,linkcolor=BrickRed,urlcolor=Orange}

\numberwithin{equation}{section}

\theoremstyle{plain}
\newtheorem{theorem}[equation]{Theorem}
\newtheorem{proposition}[equation]{Proposition}
\newtheorem{lemma}[equation]{Lemma}

\newtheorem{corollary}[equation]{Corollary}

\newtheorem{conjecture}[equation]{Conjecture}

\theoremstyle{definition}
\newtheorem{definition}[equation]{Definition}
\newtheorem{example}[equation]{Example}

\theoremstyle{remark}
\newtheorem{remark}[equation]{Remark}

\usepackage{tikz}
\usepackage{tikz-cd}
\usetikzlibrary{matrix,arrows}
\usetikzlibrary{decorations.pathmorphing}

\newcommand{\Q}{\mathbb{Q}}
\newcommand{\Z}{\mathbb{Z}}
\newcommand{\A}{\mathbb{A}}

\newcommand{\C}{\mathbb{C}}
\renewcommand{\P}{\mathbb{P}}

\DeclareMathOperator{\pr}{pr}

\DeclareMathOperator{\im}{Im}

\DeclareMathOperator{\Gr}{Gr}

\usepackage{bm}
\usepackage[left=1in,right=1in,top=1in,bottom=1in]{geometry}

\begin{document}
\title{Multiplicativity of perverse filtration for Hilbert schemes of fibered surfaces}
\author{Zili Zhang}
\thanks{The author was partially supported by NSF grant DMS-1301761.}
\maketitle
\begin{center}
{\it Mathematics Department, Stony Brook University}\\
Email: zili.zhang$@$stonybrook.edu
\end{center}

\medskip
\noindent
Let $S\to C$ be a smooth projective surface with numerically trivial canonical bundle fibered onto a curve. We prove the multiplicativity of the perverse filtration with respect to the cup product on $H^*(S^{[n]},\Q)$ for the natural morphism $S^{[n]}\to C^{(n)}$. We also prove the multiplicativity for five families of Hitchin systems obtained in a similar way and compute the perverse numbers of the Hitchin moduli spaces. We show that for small values of $n$ the perverse numbers match the predictions of the numerical version of the de Cataldo-Hausel-Migliorini $P=W$ conjecture and of  the conjecture by Hausel, Letellier and Rodriguez-Villegas.

\tableofcontents

\section{Introduction}
\noindent
\subsection{Non-abelian Hodge theory and the $P=W$ conjecture}
Given any smooth complex projective variety $X$ and any algebraic reductive group $G$, there are two natural moduli spaces associated with them, the moduli of Higgs $G$-bundles $M_D$ and the character variety $M_B$. Simpson proved that these moduli spaces are algebraic varieties, and that they are canonically diffeomorphic to each other. The diffeomorphism induces a canonical identification of the cohomology of these moduli spaces. The algebraicity endows the cohomology groups with mixed Hodge structures. Furthermore, the moduli of Higgs bundle carries a proper Hitchin map to an affine space, so that the rational cohomology $H^*(M_D)$ is endowed with the Leray filtration and with the perverse (Leray) filtration. Under the canonical isomorphism, one may compare the filtrations mentioned above.  One remarkable result is in \cite{P=W}, where de Cataldo, Hausel and Migliorini considered the case when $X$ is any curve with genus $g\ge2$ and $G=GL(2,\mathbb C)$. They proved that, via the non-abelian Hodge theorem, the perverse filtration $P$ on $H^*(M_D)$ for the Hitchin map equals the mixed Hodge theoretic weight filtration $W$ on $H^*(M_B)$; they proved that $P=W.$ There are numerous and diverse Hitchin-type moduli spaces $M_D$ that come with natural Hitchin-type maps $h: M_D \to \mathbb A$ and which have corresponding (twisted) character varieties $M_B$. It is implicitly conjectured in \cite{exchange} and \cite{P=W} that the $P=W$ phenomenon appears whenever the non-abelian Hodge theory holds.

\begin{conjecture}[The $P=W$ conjecture]\label {1.1}
Under the canonical isomorphism between the cohomology groups predicted by non-abelian Hodge theory, the perverse filtration and the weight filtration correspond to each other.
\end{conjecture}

\noindent
It had known previously that the weight filtration in mixed Hodge structure is multiplicative with respect to the cup product, which means $\cup:W_kH^i\times W_lH^j\to W_{k+l}H^{i+j}$. The perverse filtration is not multiplicative in general, even for proper maps between smooth projective varieties. So one key step in \cite{P=W} is to establish that the perverse filtration is also multiplicative.

\medskip
\noindent
This paper is devoted to prove that for five special families of Hitchin moduli spaces, the perverse filtrations on the rational cohomology are multiplicative with respect to the cup product. By using the same method, we can also prove the multiplicativity of perverse filtration on the cohomology groups of Hilbert schemes of points on elliptic K3 surfaces defined by the natural elliptic fibration.

\subsection{Perverse filtration}
In this section, we define the perverse filtration on the cohomology group defined by maps between algebraic varieties. We follow the convention in \cite{P=W}. We always work with varieties over the field of complex numbers $\C$. All cohomology groups have rational coefficients.

\medskip
\noindent
A sheaf $\mathcal{F}$ on a variety $Y$ is constructible if there is a finite partition $Y=\sqcup Y_i$ into locally closed smooth subvarieties such that each $\mathcal{F}|_{Y_i}$ is a local system of $\Q$-vector spaces. A constructible complex $K$ on a variety $Y$ is a bounded complex of sheaves whose cohomology sheaves $\mathcal{H}^i(K)$ are all constructible. Denote by $D^b_c(Y)$ the derived category of constructible complexes. The perverse truncation with the middle perversity is denoted by $^\mathfrak{p}\tau_{\le p}$. Given $K\in D^b_c(Y)$, we have a sequence of truncated complexes
$$
\cdots\to {^\mathfrak{p}}\tau_{\le p-1}K\to {^\mathfrak{p}}\tau_{\le p}K\to{^\mathfrak{p}}\tau_{\le p+1}K\to\cdots\to K,
$$
where ${^\mathfrak{p}}\tau_{\le p}K=0$ for $p\ll0$ and ${^\mathfrak{p}}\tau_{\le p}K=K$ for $p \gg0$. The perverse filtration $P$ on the hypercohomology $\mathbb{H}^*(Y,K)$ is defined as
$$
P_p\mathbb{H}^*(Y,K):=\im\left\{\mathbb{H}^*(Y,{^\mathfrak{p}}\tau_{\le p}K)\to \mathbb{H}^*(Y,K)\right\}.
$$
Let $f:X\to Y$ be a proper morphism between smooth quasi-projective varieties. The defect of the semismallness is defined to be $r(f):=\dim X\times_Y X-\dim X$. Denote by $Rf_*$ the derived push-forward functor. Then $H^*(X,\Q)=\mathbb{H}^*(Y,Rf_*\Q_X)$ is naturally endowed with a perverse filtration.  For the purpose of this paper, we want the perverse filtration to be of the type $[0,2r(f)]$, i.e. $P_{-1}=0$ and $P_{2r(f)}=H^*(X;\Q)$. To achieve this, we define
$$
P_pH^*(X;\Q):=P_p\mathbb{H}^{*-\dim X}\left(Y,Rf_*\Q_X[\dim X]\right).
$$
\noindent
Remark 2.1.3 of \cite{hodge} shows that the perverse filtration defined above is of the type $[-r(f),r(f)]$, so a shifting by $r(f)$ leads to our definition of geometric perverse filtration, which will be used throughout the paper.

\begin{definition}[\cite{P=W} 1.4.1]
Let $f:X\to Y$ be a proper morphism between smooth quasi-projective varieties. Let $r(f)=\dim X\times_Y X-\dim X$ be the defect of semismallness. Define the geometric perverse filtration as
$$
P_pH^d(X;\Q):=\im\left\{\mathbb{H}^{d-\dim X+r(f)}\left(Y, {^{\mathfrak{p}}\tau_{\le p}}Rf_*\Q_X[\dim X-r(f)]\right)\to H^d(X,\Q)\right\},
$$
where the $^{\mathfrak{p}}\tau_{\le p}$ is the truncation functor of the standard perverse $t$-structure. The filtration is of the type $[0,2r(f)]$. Let
$$
\Gr_p^PH^d(X;\Q):=P_pH^d(X;\Q)/P_{p-1}H^d(X;\Q).
$$
The perverse filtration is multiplicative with respect to the cup product if the following condition holds for any integer $p,q,i,j$.
$$
P_pH^i(X;\Q)\cup P_qH^j(X;\Q)\to P_{p+q}H^{i+j}(X;\Q).
$$ 
\end{definition}

\begin{definition}
Let $f:X\to Y$ as before. Given a cohomology class $0\neq\alpha\in H^*(X)$, define the perversity of $\alpha$, denoted as $\mathfrak{p}(\alpha)$, to be the integer such that $\alpha\in P_{\mathfrak{p}(\alpha)}H^*(X)$ and $\alpha\not\in P_{\mathfrak{p}(\alpha)-1}H^*(X)$. By our choice of perversity, the function $\mathfrak{p}$ takes value in the interval $[0,2r(f)]$. Define $\mathfrak{p}(0)=-\infty$. We say that the perverse filtration is {\it multiplicative with respect to cup product} if and only if for any two classes $\alpha,\beta\in H^*(X)$, one has $\mathfrak{p}(\alpha\cup\beta)\le \mathfrak{p}(\alpha)+\mathfrak{p}(\beta)$. 
\end{definition}

\begin{definition}
Let $f:X\to Y$ as before. A perverse decomposition for $Rf_*\Q_X$ is an isomorphism in $D^b_c(Y)$
$$
Rf_*\Q_X[\dim X-r(f)]\cong \bigoplus_{i=0}^{2r(f)}\mathcal{P}_i[-i],
$$
where $\mathcal{P}_i$ are suitable perverse sheaves. In particular, we have
$$
P_pH^*(X;\Q)=\mathbb{H}\left(\bigoplus_{i=0}^p\mathcal{P}_i[-i]\right)
$$
and
$$
\Gr_p^PH^*(X;\Q)=\mathbb{H}^*\left(\mathcal{P}_p[-p]\right).
$$
\end{definition}

\noindent
We will use the perversity function $\mathfrak{p}$ for a cohomology class $\alpha\in H^*(X)$ without mentioning the map $X\to Y$ when no confusion arises. We say {\it the perversity of} $\alpha$ for $\mathfrak{p}(\alpha)$. To simplify notation, we say {\it the perverse decomposition for the map} $f:X\to Y$ for the perverse decomposition for $Rf_*\Q_X$. We say {\it the perverse filtration for the map} $f:X\to Y$ for the perverse filtration on the cohomology group $H^*(X;\Q)$ defined by the map $f:X\to Y$.

\begin{example}
We give an example of map between varieties such that the perverse filtration is not multiplicative. Let $f:X\to Y=\P^3$ be the blowing-up at point $o$. Let $E\cong \P^2$ be the exceptional divisor. Then $r(f)=4-3=1$ and we have the perverse decomposition
$$
Rf_*\Q_X[2]=\left\{\Q_o\right\}\oplus\left\{\Q_Y[3]\right\}[-1]\oplus\left\{\Q_o\right\}[-2].
$$
Then we have $P_0H^*(X;\Q)=\mathbb{H}^*(\Q_o)=\Q E$ and $\Gr_2^PH^*(X;\Q)=\mathbb{H}^*(\Q_o[-2])=\Q H_E$, where $H_E$ is the generator of $H^2(E)$. Note that $E^2=-H_E$ in $H^*(X;\Q)$. So $\mathfrak{p}(E\cup E)=2>0=\mathfrak{p}(E)+\mathfrak{p}(E)$.
\end{example}

\begin{definition}
Let $f:X\to Y$ be as before. A basis $\alpha_1,\cdots,\alpha_k$ of cohomology group $H^*(X;\Q)$ is filtered with respect to the perverse filtration if the following property holds for any $0\le p\le 2r(f)$.
$$
P_pH^*(X)=\text{Span }\{\alpha_i\mid \mathfrak{p}(\alpha_i)\le p,1\le i\le k\}.
$$
\end{definition}

\subsection{Main Results}
In this paper, we study a beautiful and classical class of Hitchin systems $h: M_D\to \mathbb A$. They are five families of moduli spaces of parabolic Higgs bundles over $\mathbb P^1$ with marked points, labeled by the affine Dynkin diagrams $\widetilde{A_0}$, $\widetilde{D_4}$, $\widetilde{E_6}$, $\widetilde{E_7}$ and $\widetilde{E_8}$. In this setting, Theorems 4.1 and 5.1 in \cite{grochenig} states that these $M_D$ are the Hilbert schemes $S^{[n]}$ of $n$-points of five distinct smooth algebraic elliptically fibered surfaces $f:S \to \mathbb A^1$. There are, for each of the five surfaces and for each $n\ge 1$, Hitchin maps $h:M_D=S^{[n]}\to \A^n$, hence a perverse filtration $P$ on the cohomology groups $H^*(M_D)$. The construction of Hitchin map $h$ is analogous to the one that starts with an elliptic $K3$ surface $f:S\to \P^1$ and yields the natural map $h:S^{[n]}\to \P^n$. The Main result of the paper is the multiplicativity of these perverse filtrations.

\begin{theorem}[Theorem \ref{multiplicative open}]
For the five families of Hitchin systems defined in section 5.1, the perverse filtration on the rational cohomology $H^*(M_D;\Q)$ defined by the map $h:M_D\to \A$ is multiplicative.
\end{theorem}

\noindent
In the proof, we develop a systematic framework to describe the perverse filtration on $H^*(M_D)$ in terms of the one on $H^*(S)$ defined by the map $S\to \A^1$. We use the decomposition theorem of Beilinson-Bernstein-Deligne-Gabber \cite{BBD} as our main tool to decompose $Rf_*\Q_S$. By using the explicit geometry, we pick a very special basis of $H^*(S)$ which is filtered respect to the perverse filtration, and use it to produce a filtered basis of $H^*(M_D)$. The key step is the determination of the precise perversity of the class of the small diagonals in the product $S^n$; the general bounds for these perversities are too weak for the problem, and we have to improve upon them by using the special geometry. The description of Lehn in \cite{lehn}, of the cohomology ring of $H^*(S^{[n]})$ is a key ingredient in our approach. Since $M_D$ is not compact, we prove that it is also valid for all the moduli spaces $M_D$ we are working with. 

\medskip
\noindent
By using similar techniques, we may also prove the multiplicativity of perverse filtration for Hilbert schemes of projective surfaces with numerically trivial canonical bundle. In fact, one can start with any smooth quasi-projective surfaces $f:S\to C$ fibered over a curve and obtain a map $f^{[n]}:S^{[n]}\to C^{(n)}$, where $C^{(n)}$ is the $n$-th symmetric product of $C$. In this case, since we don't have explicit description of the map, we have to use the ``relative" Hodge-Riemann bilinear relations, due to de Cataldo and  Migliorini \cite{hodge}, to ``calculate" the perverse filtration and to produce special basis for the cohomology groups which are adapted to our problem. We have the following result.

\begin{theorem}[Theorem \ref{multiplicative^n}]
Let $f:S\to C$ be a surjective morphism from a smooth projective surface with numerically trivial canonical bundle to a smooth projective curve. Then the perverse filtration of $H^*(S^{[n]};\Q)$ defined by the morphism $f:S^{[n]}\to C^{(n)}$ is multiplicative, namely, we have
$$
P_pH^*(S^{[n]};\Q)\cup P_{p'}H^*(S^{[n]};\Q)\subset P_{p+p'}H^*(S^{[n]};\Q).
$$
\end{theorem}

\noindent
As a byproduct of our formalism, we can prove that if there is a pair of smooth surfaces $S_P$ and $S_W$, such that the perverse filtration  (defined by some proper map $h:S_P\to \A^1$) on the cohomology $H^*(S_P)$ corresponds naturally to the weight filtration in the mixed Hodge structure on the cohomology of $H^*(S_W)$, then this correspondence induces an identification between the perverse filtration on $H^*(S_P^{[n]})$ and the weight filtration on $H^*(S_W^{[n]})$.  This generalizes \cite{exchange} Theorem 4.1.1.

\medskip
\noindent
There is a numerical version of the $P=W$ conjecture, namely instead of requiring the filtrations to correspond via the non-abelian Hodge theorem, one only requires the dimensions of the graded pieces to be the same. Conjectures in \cite{perverse number} and \cite{hodge number} predict the perverse numbers and mixed Hodge numbers for the moduli of parabolic Higgs bundles over curves with marked points and the corresponding character varieties. In our five families of Hitchin systems, we compute the perverse filtration explicitly, and also the perverse numbers. 

\begin{theorem}[Theorem \ref{8.1}]
Let $f:S\to \A^1$ be the $n=1$ case of the five families of Hitchin systems defined in section 5.1. Denote the perverse numbers by $p^{i,j}_n=\dim \Gr_i^PH^j(S^{[n]})$. Let the perverse Poincar\'e polynomial be $P_n(q,t)=\sum_{i,j}p^{i,j}_nq^it^j$. Then in the $\widetilde{A_0}$ case, the generating series is
$$
\sum_{n=0}^\infty s^nP_n(q,t)=\prod_{m=1}^\infty\frac{(1+s^mq^{m}t^{2m-1})^2}{(1-s^mq^{m-1}t^{2m-2})(1-s^mq^{m+1}t^{2m})}.
$$
In the other four cases $\widetilde{D_4}$, $\widetilde{E_6}$, $\widetilde{E_7}$ and $\widetilde{E_8}$, the generating series are
$$
\sum_{n=0}^\infty s^nP_n(q,t)=
\prod_{m=1}^\infty\frac{1}{(1-s^mq^{m-1}t^{2m-2})(1-s^mq^mt^{2m})^k(1-s^mq^{m+1}t^{2m})}
$$
where $k$ is an integer defined in Proposition \ref{decomposition 5 cases}.
\end{theorem} 

\noindent
Using the explicit description of the corresponding character varieties for $n=1$ in \cite{character}, we prove the full version of the $P=W$ conjecture for $S\to  \mathbb{A}^1$ in each of our five cases. For $n\ge2$, little is known about the corresponding character varieties. However, there are conjectures concerning the shape of the filtration $W$ on $H^*(M_B)$ in \cite{hodge number}. Mathematica computations show that for small $n$, the perverse numbers obtained in our theorem match the conjectural mixed Hodge numbers in \cite{hodge number}.

\medskip
\noindent
{\bf Acknowledgements.}
I am grateful to my advisor Mark de Cataldo who suggested to me this project and explained the Hodge-theoretic interpretation of the decomposition theorem to me. I thank Dingxin Zhang for discussion on numerous technical details. I thank Lie Fu for discussions on the diagonal in the Cartesian product of K3 surfaces, which motivated my estimation of the perversity of diagonals. I thank Zhiwei Yun for pointing out an inaccuracy in an earlier version of the paper. I thank  Luca Migliorini, Jingchen Niu, Carlos Simpson, Qizheng Yin and Letao Zhang for useful conversations. I would also like to thank the referee for the careful review and the valuable comments.

\section{Functoriality of the perverse filtrations}
\noindent
In this chapter we prove that external tensor products, symmetric products and alternating products of perverse sheaves are perverse. We also show we may describe the perverse filtrations for $f^n:X^n\to Y^n$ and $f^{(n)}:X^{(n)}\to Y^{(n)}$ in terms of the perverse filtration for $f:X\to Y$. We use terms ``perversity'', ``perverse decomposition''and ``perverse filtration'' under the convention defined in section 1.2.

\subsection{External tensor product}
\begin{proposition}\label{product}
Let $f_1:X_1\to Y_1$, $f_2:X_2\to Y_2$ be two proper morphisms between smooth quasi-projective varieties. Let $r(f)$ denote the defect of semismallness of $f$ defined in section 1.2. Let $\mathcal{F}_i$ and $\mathcal{G}_j$ be suitable perverse sheaves and 
$$
Rf_{1,*}\Q_{X_1}[\dim X_1-r(f_1)]\cong \bigoplus_{i=0}^{2r(f_1)} \mathcal{F}_i[-i],
$$

$$
Rf_{2,*}\Q_{X_2}[\dim X_2-r(f_2)]\cong\bigoplus_{j=0}^{2r(f_2)} \mathcal{G}_j[-j]
$$ 
be the perverse decompositions for map $f_1$, $f_2$, respectively. Then

$$
R(f_1\times f_2)_*\Q_{X_1\times X_2}[\dim X_1\times X_2-r(f_1\times f_2)]\cong\bigoplus_{i,j}\mathcal{F}_i\boxtimes \mathcal{G}_j[-i-j]
$$
is a perverse decomposition for the proper map $f_1\times f_2:X_1\times X_2\to Y_1\times Y_2$. In particular, for $\alpha_1\in H^*(X_1)$, $\alpha_2\in H^*(X_2)$, we have $\mathfrak{p}(\alpha_1\otimes\alpha_2)=\mathfrak{p}(\alpha_1)+\mathfrak{p}(\alpha_2)$, where $\alpha_1\otimes\alpha_2$ is viewed as a cohomology class in $H^*(X_1\times X_2)$, and the perverse filtration is defined by the map $f_1\times f_2:X_1\times X_2\to Y_1\times Y_2$.
\end{proposition}

\begin{proof}
We first check that $r(f_1\times f_2)=r(f_1)+r(f_2)$. In fact, by the universal property of fibered product, we have
$$
(X_1\times X_2)\times_{(Y_1\times Y_2)}(X_1\times X_2)=(X_1\times_{Y_1}X_1)\times(X_2\times_{Y_2}X_2).
$$ 
So 
\begin{equation*}
\begin{array}{rrl}
r(f_1\times f_2)&=&\dim (X_1\times X_2)\times_{(Y_1\times Y_2)}(X_1\times X_2)-\dim Y_1\times Y_2\\
&=&\dim X_1\times_{Y_1}X_1+\dim X_2\times_{Y_2}X_2-\dim Y_1-\dim Y_2\\
&=& r(f_1)+r(f_2).
\end{array}
\end{equation*}
To prove the isomorphism, it suffices to note that $f_1$, $f_2$ and $f_1\times f_2$ are all proper, so $Rf_*=Rf_!$. By the K\"unneth formula (see exercise II.18 of \cite{sheaves}), we have 
$$
\begin{array}{rl}
&\displaystyle R(f_1\times f_2)_*\Q_{X_1\times X_2}[\dim X_1+\dim X_2-r(f_1)-r(f_2)]\\
=&\displaystyle R(f_1\times f_2)_*\Q_{X_1}\boxtimes\Q_{X_2}[\dim X_1+\dim X_2-r(f_1)-r(f_2)]\\
=&\displaystyle Rf_{1,*}\Q_{X_1}[\dim X_1-r(f_1)]\boxtimes Rf_{2,*}\Q_{X_2}[\dim X_2-r(f_2)]\\
\cong&\displaystyle \bigoplus_{i,j}\mathcal{F}_i\boxtimes \mathcal{G}_j[-i-j]
\end{array}
$$
\noindent
By Proposition 10.3.6 (i)(ii) of \cite{sheaves}, the external tensor product $\mathcal{F}_i\boxtimes\mathcal{G}_j$ is perverse. Therefore this gives a perverse decomposition.  To check the perversity is additive with respect to tensor product is basically by definition as follows. Let $p_1=\mathfrak{p}(\alpha_1)$, $p_2=\mathfrak{p}(\alpha_2)$. Recall that by definition of geometric perversity, we have  
$$
\alpha_1\in \mathbb{H}\left(\bigoplus_{i\le p_1}\mathcal{F}_i[-i]\right),\alpha_2\in \mathbb{H}\left(\bigoplus_{j\le p_2}\mathcal{G}_j[-j]\right)
$$

So
\begin{equation*}
\begin{split}
\alpha_1\otimes\alpha_2\in&\mathbb{H}\left(\bigoplus_{i\le p_1,j\le p_2}\mathcal{F}_i\boxtimes\mathcal{G}_j[-i-j]\right)\\
\subset&\mathbb{H}\left(\bigoplus_{i+j\le p_1+p_2}\mathcal{F}_i\boxtimes\mathcal{G}_j[-i-j]\right)\\
=&P_{p_1+p_2}H^*(X_1\times X_2)
\end{split}
\end{equation*}
This shows that $\mathfrak{p}(\alpha_1\otimes\alpha_2)\le \mathfrak{p}(\alpha_1)+\mathfrak{p}(\alpha_2)$. On the other hand, $\mathfrak{p}(\alpha_k)=p_k$ means that $\alpha_k\neq0\in \Gr_{p_k}^PH^*(X_k)$. So 
$$
0\neq\alpha_1\otimes\alpha_2\in \Gr_{p_1}^PH^*(X_1)\otimes \Gr_{p_2}^PH^*(X_2)\subset \Gr_{p_1+p_2}^PH^*(X_1\times X_2)
$$
This shows that $\mathfrak{p}(\alpha_1\otimes\alpha_2)=\mathfrak{p}(\alpha_1)+\mathfrak{p}(\alpha_2)$.
\end{proof}

\begin{corollary}\label{cartesian}
Let $f:X\to Y$ be a proper map between smooth quasi-projective varieties. Then the perverse filtration for the product map $f^n:X^n\to Y^n$ can be described as
$$
P_pH^*(X^n;\Q)={\rm{Span }}\{\alpha_1\otimes\cdots\otimes\alpha_n\mid \mathfrak{p}(\alpha_1)+\cdots+\mathfrak{p}(\alpha_n)\le p\},
$$
where $\alpha_i\in H^*(X;\Q)$ for $i=1,\cdots,n$.
\end{corollary}

\subsection{Symmetric and alternating product}
\noindent
In this section, we give an explicit description of the $\mathfrak{S}_n$-action on the $n$-fold external tensor product of a bounded complex, where $\mathfrak{S}_n$ is the symmetric group of $n$ elements. We use the action to define the symmetric product in derived category of constructible sheaves, and show that the symmetric product of a perverse sheaf is still perverse. Our method is similar to the one in \cite{mhm}. Let $X^{(n)}=X^n/\mathfrak{S}_n$ denote the $n$-th symmetric product of $X$.\\
\noindent
\begin{definition}
Let $K^\bullet_i$ be a bounded complex of constructible sheaves on a complex quasi-projective variety $X_i$, for $1\le i\le n$. Then the $n$-fold external tensor product $\boxtimes_{i=1}^n K^\bullet_i$ on $\prod_{i=1}^n {X_i}$ is defined as follows.
\begin{enumerate}
\item{The $j$-th component is $\bigoplus_{\sum k_i=j}\boxtimes_{i=1}^n K^{k_i}_i$. }
\item{The differential is $\sum_{i=1}^n{(-1)^{k_1+\cdots+k_{i-1}}d_i}$ on the summand $\boxtimes_{i=1}^n K^{k_i}_i$, where $d_i$ is induced by the differential of $K_i$.}
\end{enumerate}
\end{definition}

\begin{definition}
Let $K^\bullet_i$  and $X_i$ as above. Then there is a natural $\mathfrak{S}_n$-action on $\boxtimes_{i=1}^n K^\bullet_i$ by:
$$
\sigma^{\#}:\boxtimes_{i=1}^n K^\bullet_i\xrightarrow{\sim}\sigma_*\left(\boxtimes_{i=1}^n K^\bullet_{\sigma(i)}\right)
$$
which is defined, for $m_i\in K_i^{p_i}$, by
$$
\boxtimes_{i=1}^nm_i\mapsto (-1)^{\nu(\sigma,p)}\sigma_*\left(\boxtimes_{i=1}^nm_{\sigma(i)}\right)
$$
where $\nu(\sigma,p)=\sum_{i<j,\sigma(j)<\sigma(i)}p_ip_j$.
\end{definition}

\begin{definition}
Let $X$ be a complex quasi-projective variety. Let $q:X^n\to X^{(n)}$ be the quotient map. For a bounded complex of constructible sheaves $K$ on $X$, we define the symmetric product and alternating product as
\begin{equation*}
\begin{array}{rrl}
K^{(n)}&=&\left(Rq_*K^{\boxtimes n}\right)^{\mathfrak{S}_n}\\
K^{\{n\}}&=&\left(Rq_*K^{\boxtimes n}\right)^{\text{sign}-\mathfrak{S}_n}
\end{array}
\end{equation*}
where 
$$
(-)^{\mathfrak{S}_n}=\frac{1}{n!}\sum_{\sigma\in\mathfrak{S}_n}Rq_*(\sigma^\#)
$$
is the symmetrizing projector and 
$$
(-)^{\text{sign}-\mathfrak{S}_n}=\frac{1}{n!}\sum_{\sigma\in\mathfrak{S}_n}(-1)^{\text{sign}(\sigma)}Rq_*(\sigma^\#)$$
is the alternating projector. Here we use the fact that $\mathfrak{S}_n$ acts trivially on $X^{(n)}$.
\end{definition}

\begin{remark}\label{quotient}
By \cite{macdonald} (1.1), we have the following canonical isomorphisms.
$$
H^*(X^{(n)},K^{(n)})=H^*(X^n,K^{\boxtimes n})^{\mathfrak{S}_n}=\bigoplus_{i+j=n}Sym^i H^{even}(X,K)\bigotimes \bigwedge^jH^{odd}(X,K)
$$
$$
H^*(X^{(n)},K^{\{n\}})=H^*(X^n,K^{\boxtimes n})^{\text{sign}-\mathfrak{S}_n}=\bigoplus_{i+j=n}\bigwedge^i H^{even}(X,K)\bigotimes Sym^jH^{odd}(X,K)
$$
Furthermore, we have
$$
(K[a])^{(n)}=
\begin{cases}
K^{(n)}[na]& \text{if }a\text{ is even.}\\
K^{\{n\}}[na]& \text{if }a\text{ is odd.}
\end{cases}
$$
\end{remark}

\begin{proposition}\label{symmetric}
Let $\mathcal{P}$ be a perverse sheaf on $X$. Let $q:X^n\to X^{(n)}$ be the quotient map. Then $\mathcal{P}^{(n)}$ and $\mathcal{P}^{\{n\}}$ are perverse sheaves on $X^{(n)}$.
\end{proposition}

\begin{proof}
By Proposition \ref{product}, $\mathcal{P}^{\boxtimes n}$ is perverse. Since the map $q:X^n\to X^{(n)}$ is finite, $Rq_*K^{\boxtimes n}$ is perverse (\cite{BBD} Corollaire 2.2.6 (i) ). It suffices to prove that the invariant part and the alternating part under the $\mathfrak{S}_n$-action are both perverse. By the definition of the projectors, we have 
$$
(Rq_*\mathcal{P}^{\boxtimes n})^{\mathfrak{S}_n}\to Rq_*\mathcal{P}^{\boxtimes n}\to(Rq_*\mathcal{P}^{\boxtimes n})^{\mathfrak{S}_n}
$$
$$
(Rq_*\mathcal{P}^{\boxtimes n})^{\text{sign}-\mathfrak{S}_n}\to Rq_*\mathcal{P}^{\boxtimes n}\to(Rq_*\mathcal{P}^{\boxtimes n})^{\text{sign}-\mathfrak{S}_n}
$$
where both compositions are the identity. This means that $(Rq_*\mathcal{P}^{\boxtimes n})^{\mathfrak{S}_n}$ and $(Rq_*\mathcal{P}^{\boxtimes n})^{\text{sign}-\mathfrak{S}_n}$ are both direct summands of $Rq_*\mathcal{P}^{\boxtimes n}$ in the bounded derived category of constructible sheaves. The proposition holds due to the following lemma.
\end{proof}

\begin{lemma}
Let $\mathcal{P}$ be a perverse sheaf on $X$. Suppose that $\mathcal{P}=K\oplus K'$ holds in $D^b_c(X)$, the bounded derived category of constructible sheaves. Then $K$ is perverse.
\end{lemma}

\begin{proof}
The cohomology sheaf satisfies $\mathcal{H}^i \mathcal{P}=\mathcal{H}^iK\oplus\mathcal{H}^iK'$. Therefore $\text{Supp }\mathcal{H}^iK\subset \text{Supp } \mathcal{H}^i\mathcal{P}$, and thus $\dim \text{Supp }\mathcal{H}^iK\le\dim \text{Supp }\mathcal{H}^i \mathcal{P}\le -i$. This proves the support condition (4.0.1') in \cite{BBD}. Noting that $\mathcal{P}^\vee=K^\vee\oplus (K')^\vee$, the cosupport condition follows similarly.
\end{proof}

\subsection{Perverse filtration of symmetric products}
\noindent
In this section, we show that the perverse filtration for a symmetric product a morphism $f^{(n)}:X^{(n)}\to Y^{(n)}$ is compatible with the perverse filtration of the corresponding cartesian product $f^n:X^n\to Y^n$. We also use the symmetric product and the alternating product for perverse sheaves to give a perverse decomposition for the symmetric product of maps.

\begin{lemma} \label{5.1}
Let $X$ be a smooth quasi-projective variety. Let $q:X^n\to X^{(n)}$. Let $K_i\in D^b_c(X)$, $i=1,\cdots,n$. Then $\mathfrak{S}_n$ acts on 
$$
\widetilde{K}=\bigoplus_{\sigma\in\mathfrak{S}_n}K_{\sigma(1)}\boxtimes \cdots \boxtimes K_{\sigma(n)}
$$
as an endomorphism. Furthermore, $(Rp_*\widetilde{K})^{\mathfrak{S}_n}\cong Rp_*(K_1\boxtimes\cdots\boxtimes K_n)$. More generally, let $\bold{k}=(k_1,\cdots,k_m)$ be a $m$-tuple with $k_1+\cdots+k_m=n$. Let 
$$
\mathfrak{S}_{\bold{k}}:=\{\sigma:[n]\to [m]\mid |f^{-1}(i)|=k_i\},
$$
where $[n]$ denotes the set $\{1,\cdots,n\}$. Let $q_\bold{k}:X^{(k_1)}\times\cdots X^{(k_m)}\to X^{(n)}$ be the natural map. Then $\mathfrak{S}_n$ acts on 
$$
\widetilde{K}_{\bold{k}}:=\bigoplus_{\sigma\in\mathfrak{S}_{\bold{k}}}K_{\sigma(1)}\boxtimes \cdots \boxtimes K_{\sigma(n)},
$$
and we have 
$$
(Rq_*\widetilde{K}_{\bold{k}})^{\mathfrak{S}_n}\cong Rq_{\bold{k},*}(K_1^{(k_1)}\boxtimes\cdots\boxtimes K_m^{(k_m)}).
$$
Similarly, for the alternating part we have 
$$
(Rq_*\widetilde{K}_{\bold{k}})^{\text{sign-}\mathfrak{S}_n}\cong Rq_{\bold{k},*}(K_1^{\{k_1\}}\boxtimes\cdots\boxtimes K_m^{\{k_m\}}).
$$
\end{lemma}

\begin{proof}
Note that $\mathfrak{S}_n$ acts on $\widetilde{K}$ by permuting the direct summands (up to sign). The invariant part of the push-forward is determined by any one of its summands. 
\end{proof}

\begin{proposition}\label{5.2}
Let $X$ be a smooth quasi-projective variety. Let $q:X^n\to X^{(n)}$ be the natural quotient map. Let $K=\oplus_{i=1}^m K_i\in D^b_c(X)$. Then we have the expansion
$$
K^{(n)}\cong\bigoplus_{\bold{k}}Rq_{\bold{k},*}(K_1^{(k_1)}\boxtimes\cdots\boxtimes K_n^{(k_n)}).
$$ 
and
$$
K^{\{n\}}\cong\bigoplus_{\bold{k}}Rq_{\bold{k},*}(K_1^{\{k_1\}}\boxtimes\cdots\boxtimes K_n^{\{k_n\}}).
$$ 

\end{proposition}

\begin{proof}
By Lemma \ref{5.1}, we have
$$
\begin{array}{rll}
K^{(n)}&=&(Rq_*K^{\boxtimes n})^{\mathfrak{S}_n}\\
&=&\displaystyle\left(Rq_*\bigoplus_{1\le i_1,\cdots,i_n\le m} K_{i_1}\boxtimes\cdots\boxtimes K_{i_n}\right)^{\mathfrak{S}_n}\\
&=&\displaystyle\left(Rq_*\bigoplus_{\bold{k}}\widetilde{K}_{\bold{k}}\right)^{\mathfrak{S}_n}\\
&=&\displaystyle\bigoplus_{\bold{k}}\left(Rq_*\widetilde{K}_{\bold{k}}\right)^{\mathfrak{S}_n}\\
&=&\displaystyle\bigoplus_{\bold{k}}Rq_{\bold{k},*}\left(K_1^{(k_1)}\boxtimes\cdots\boxtimes K_m^{(k_m)}\right).
\end{array}
$$
\end{proof}

\begin{lemma}\label{5.3}
Let $f:X\to Y$ be a proper morphism between smooth quasi-projective varieties. We have the following commutative diagram
$$
\begin{tikzcd}
X^n\ar{r}{q}\ar{d}{f^n}& X^{(n)}\ar{d}{f^{(n)}}\\
Y^n\ar{r}{q}&Y^{(n)}
\end{tikzcd}
$$
 Let $K\in D^b_c(X)$, then $Rf^{(n)}_*K^{(n)}\cong (Rf_*K)^{(n)}$. 
\end{lemma}

\begin{proof}
\begin{equation*}
\begin{array}{rl}
(Rf_*K)^{(n)}&\cong\left(Rq_*(Rf_*K)^{\boxtimes n}\right)^{\mathfrak{S}_n}\cong\left(Rq_*Rf_*K^{\boxtimes n}\right)^{\mathfrak{S}_n}\\
&\cong \left(Rf^{(n)}_*Rq_*K^{\boxtimes n}\right)^{\mathfrak{S}_n}\cong Rf^{(n)}_*\left(Rq_*K^{\boxtimes n}\right)^{\mathfrak{S}_n}\cong Rf^{(n)}_*K^{(n)}.
\end{array}
\end{equation*}
\end{proof}

\begin{proposition}\label{5.4}
Let $f:X\to Y$ be a proper morphism between smooth quasi-projective varieties. Let
$$
Rf_*\Q_X[\dim X-r(f)]\cong\bigoplus_{i=0}^{2r(f)}\mathcal{P}_i[-i]
$$
be the perverse decomposition, where $\mathcal{P}_i$ are perverse sheaves on $Y$. Then the perverse decomposition of the map $f^{(n)}:X^{(n)}\to Y^{(n)}$ is given as follows. When $\dim X-r(f)$ is even, then
$$
\begin{array}{rrl}
& &Rf^{(n)}_*\Q_{X^{(n)}}[n(\dim X-r(f))]\\
&\cong&\displaystyle\bigoplus_{\bold{k}}Rq_{\bold{k},*}\left(\mathcal{P}_0^{(k_0)}\boxtimes(\mathcal{P}_1[-1])^{(k_1)}\boxtimes\cdots\boxtimes(\mathcal{P}_{2r(f)}[-2r(f)])^{(k_{2r(f)})}\right)\\
&\cong&\displaystyle\bigoplus_{\bold{k}}Rq_{\bold{k},*}\left(\mathcal{P}_0^{(k_0)}\boxtimes\mathcal{P}_1^{\{k_1\}}\boxtimes\cdots\boxtimes\mathcal{P}_{2r(f)}^{(k_{2r(f)})}\right)\left[-\sum_{i=0}^{2r(f)}{ik_i}\right].\\
\end{array}
$$

\noindent
When $\dim X-r(f)$ is odd, then
$$
\begin{array}{rrl}
& &Rf^{(n)}_*\Q_{X^{(n)}}[n(\dim X-r(f))]\\
&\cong&\displaystyle\bigoplus_{\bold{k}}Rq_{\bold{k},*}\left(\mathcal{P}_0^{\{k_0\}}\boxtimes(\mathcal{P}_1[-1])^{\{k_1\}}\boxtimes\cdots\boxtimes(\mathcal{P}_{2r(f)}[-2r(f)])^{\{k_{2r(f)}\}}\right)\\
&\cong&\displaystyle\bigoplus_{\bold{k}}Rq_{\bold{k},*}\left(\mathcal{P}_0^{\{k_0\}}\boxtimes\mathcal{P}_1^{(k_1)}\boxtimes\cdots\boxtimes\mathcal{P}_{2r(f)}^{\{k_{2r(f)}\}}\right)\left[-\sum_{i=0}^{2r(f)}{ik_i}\right].\\
\end{array}
$$
\end{proposition}

\begin{proof}
By the canonical isomorphism $(\Q_X)^{(n)}=\Q_{X^{(n)}}$, Remark \ref{quotient} and Lemma \ref{5.3}, we have
$$
Rf^{(n)}_*\Q_{X^{(n)}}[n(\dim X-r(f))]=
\begin{cases}
\left(Rf_*\Q_X[\dim X-r(f)]\right)^{(n)}& \text{if }\dim X-r(f) \text{ is even}.\\
\left(Rf_*\Q_X[\dim X-r(f)]\right)^{\{n\}}& \text{if }\dim X-r(f) \text{ is odd}.\\
\end{cases}
$$
Then we use Proposition \ref{5.2} to obtain the isomorphism. Using Proposition \ref{product}, Proposition \ref{symmetric} and the fact that the projection $q_\bold{k}$ is finite, we know this isomorphism is indeed a perverse decomposition.
\end{proof}

\noindent 
Although the perverse decomposition for the symmetric product is somewhat complicated, the perverse filtration is much simpler. It is compatible with the one for the cartesian product as one may expect. To see this, we need the following lemma.

\begin{lemma}\label{5.5}
Let $f:X\to Y$ be a proper morphism between smooth quasi-projective varieties. Then
$$
{^\mathfrak{p}}\tau_{\le p}\left((Rf_*\Q_X)^{(n)}\right)=\left(Rq_*\left({^\mathfrak{p}}\tau_{\le p}( Rf_*\Q_X)^{\boxtimes n}\right)\right)^{\mathfrak{S}_n}
$$
\end{lemma}

\begin{proof}
Note that the $\mathfrak{S}_n$-invariant part is a direct summand, so it commutes with the functor ${^\mathfrak{p}}\tau_{\le p}$. Furthermore, the quotient map $q$ is finite, hence $Rq_*$ is $t$-exact. 
\end{proof}

\begin{proposition}\label{5.6}
Let $f:X\to Y$ be a proper morphism between smooth quasi-projective varieties. Under the isomorphism 
\begin{equation*}
H^*\left(X^{(n)}\right)=\left(H^*(X^n)\right)^{\mathfrak{S}_n},
\end{equation*}
the perverse filtration can be identified as 
$$
P_pH^*\left(X^{(n)}\right)=\left(P_p H^*(X^n)\right)^{\mathfrak{S}_n},
$$ 
where the perversity on the right side is defined for the map $f^n:X^n\to Y^n.$
\end{proposition}

\begin{proof}
By Lemma \ref{5.3} and Lemma \ref{5.5} , we have
$$
\begin{array}{rrl}
{^{\mathfrak{p}}\tau_{\le p}}Rf^{(n)}_*\Q_{X^{(n)}}&=& \displaystyle {^{\mathfrak{p}}\tau_{\le p}}(Rf_*\Q_X)^{(n)} \\
&=&\displaystyle \left(Rq_*{^{\mathfrak{p}}\tau_{\le p}}(Rf_*\Q_X)^{\boxtimes n}  \right)^{\mathfrak{S}_n}.\\
\end{array}
$$
After taking cohomology, we have
$$
\begin{array}{rrl}
P_pH^*(X^{(n)})&=&\mathbb H\left(Y^{(n)}, {^{\mathfrak{p}}\tau_{\le p}}Rf^{(n)}_*\Q_{X^{(n)}}\right)\\
&=&\displaystyle \left(P_pH^*(X^n)\right)^{\mathfrak{S}_n}.\\
\end{array}
$$
So the result follows.
\end{proof}

\section{Perversity of the diagonal}
\noindent
We will prove a technical result about the diagonal embedding, which is true for any smooth projective variety. The result is crucial in the proof of the multiplicativity of the perverse filtration for Hilbert schemes of smooth projective surfaces. In fact, by using Lehn's description of ring structure of Hilbert scheme of surfaces with numerically trivial canonical bundle, the perversity estimation of the diagonals is equivalent to the multiplicativity of perversity filtration for the Hilbert schemes.

\subsection{Filtered basis for cohomology groups}
\noindent
We choose and fix a basis for the cohomology group with the following properties, which is crucial in the perversity estimation of the diagonal embedding.
 
\begin{proposition}\label{basis}
Let $f:X\to Y$ be a morphism between projective varieties. Let $k(p,d)=\dim \Gr_p^PH^d(X)$.
There exists an $\Q$-basis 
$$
B=\{\beta^d_{p,i}\mid 0\le d\le 2\dim X,0\le p\le 2r(f),1\le i\le k(p,d)
\}\subset H^*(X)
$$ 
with the following properties:
\begin{enumerate}
\item{
$\beta^d_{p,i}\in P_pH^d(X)$. $\{\overline{\beta^d_{p,1}},\cdots,\overline{\beta^d_{p,k(p,d)}}\}$ is a basis of $\Gr_p^PH^d(X)$, where $\overline{\beta^d_{p,i}}$ is the image of $\beta^d_{p,i}$ under the natural quotient map 
$P_pH^d(X)\to \Gr_p^PH^d(X)$.  
}
\item{
The basis $B$ is signed orthonormal in the following sense.
\begin{equation*}
\langle\beta^d_{p,i},\beta^{d'}_{p',j}\rangle=
\begin{cases}
\pm 1 & d+d'=2\dim X,\text{ }p+p'=2r(f)\text{ and }i=j,\\
0 & otherwise.
\end{cases}
\end{equation*}
where $-1$ can only appear when $d=d'=\dim X$, $p=p'=r(f)$ and $i=j$.
}
\end{enumerate} 
In particular, if $A=\{\alpha^d_{p,i}\}$ is the dual basis of $B$ with respect to the Poincar\'e pairing, then we have
$$
\mathfrak{p}(\alpha^d_{p,i})+\mathfrak{p}(\beta^d_{p,i})=2r(f).
$$
\end{proposition}

\noindent
To prove Proposition \ref{basis}, we need the following two results.

\begin{lemma}[\cite{hodge} Version 1,  Lemma 2.9.1]\label{zero}
Let $f:X\to Y$ be a proper map between smooth projective varieties. The Poincar\'e paring $$P_pH^{d}(X)\times P_{p'}H^{2\dim X-d}(X)\to \Q$$ is trivial for $p+p'<2r(f)$.
\end{lemma}

\begin{proof}
Denote by $D^b_c(Y)$ the bounded derived category of constructible sheaf on $Y$. Let 
$\epsilon:Rf_*\Q_X[n]\to \mathcal{D}(Rf_*\Q_X[n])$ be the duality ismorphism. For every $d$, the map $\epsilon$ defines the non-degenerate Poincar\'e pairing 
$$
\int_X:H^d(X)\times H^{2\dim X-d}(X)\to\Q.
$$
So to prove the pairing is trivial, it suffices to prove that the following composition is $0$:
$$
^{\mathfrak{p}}\tau_{\le p-r(f)}Rf_*\Q_X[n]\to Rf_*\Q_X[n]\xrightarrow{\epsilon}\mathcal{D}\left(Rf_*\Q_X[n]\right)\to\mathcal{D}(^{\mathfrak{p}}\tau_{\le p'-r(f)}Rf_*\Q_X[n]).
$$
By our choice of geometric perversity, the dualizing functor $\mathcal{D}$ satisfies $$\mathcal{D}\left(D^b_c(Y)^{\le p'-r(f)}\right)\subset D^b_c(Y)^{\ge r(f)-p'}.$$ By the axioms of a $t$-structure, $Hom(D^b_c(Y)^{\le p-r(f)},D^b_c(Y)^{\ge r(f)-p'})=0$, since $p-r(f)<r(f)-p'$. So the composition is $0$.
\end{proof}

\begin{lemma}\label{nondeg}
The pairing induced by the Poincar\'e pairing 
$$
\Gr_p^PH^d(X)\times \Gr_{2r(f)-p}^PH^{2n-d}(X)\to\Q
$$
is non-degenerate.
\end{lemma}

\begin{proof}
By Lemma \ref{zero}, the pairing is well-defined. The non-degeneracy is due to Theorem 2.1.4 and Corollary 2.1.8 of \cite{hodge}.
\end{proof}

\begin{proof}[Proof of Proposition \ref{basis}]
Denote $B^d_p=\{\beta\in B\mid \mathfrak{p}(\beta)=p,\beta\in H^d(X)\}$. We construct $B^d_p$ using a Gram-Schmidt type argument. We perform the construction inductively in the lexicographical order of pairs $(p,d)$. \\
Induction base: for $(p,d)\prec(r(f),\dim X)$, pick any basis of $\Gr_p^PH^d(X)$, and lift them to get $B^d_p$. For $(p,d)=(r(f),\dim X)$, by Lemma \ref{nondeg}, the self-intersection form is nondegenerate, so we may pick a basis such that the intersection matrix is diagonal and has only $\pm1$ on the diagonal. Denote any lift of this basis by $B^{\dim X}_{r(f)}$.\\
We are now going to find $B_{p,d}=\{\beta^d_{p,1},\cdots,\beta^d_{p,k(p,d)}\}$, assuming that all cases below $(p,d)$ are done. To simplify notation, we let $e=2\dim X-d$ and $q=2r(f)-p$. Note that $(q,e)\prec(p,d)$. By Lemma \ref{zero}, the pairing $\Gr_p^PH^d(X)\times \Gr_q^PH^e(X)$ is non-degenerate, so we may pick a basis $\widetilde{B^d_p}=\left\{\widetilde{\beta^d_{p,1}},\cdots,\widetilde{\beta^d_{p,k(p,d)}}\right\}$ such that the matrix of this bilinear pairing is the identity matrix with respect to the bases $B^e_q$ and $\widetilde{B^d_p}$. Modify $\widetilde{B^d_p}$ by setting

\begin{equation}\label{gram}
\left(
\begin{array}{ccc}
\beta^d_{p,1}\\
\cdots\\
\beta^d_{p,k(p,d)}
\end{array}
\right)
=
\left(
\begin{array}{ccc}
\widetilde{\beta^d_{p,1}}\\
\cdots\\
\widetilde{\beta^d_{p,k(p,d)}}
\end{array}
\right)
+
\sum_{i=q+1}^{p-1} A_i
\left(
\begin{array}{ccc}
\beta^d_{i,1}\\
\cdots\\
\beta^d_{i,k(i,d)}
\end{array}
\right),
\end{equation}
where the $A_i$ are $k(p,d)\times k(i,d)$ matrices of rational numbers to be determined. The condition that the $A_i$ need to satisfy is slightly different when $d<\dim X$, $d>\dim X$ and $d=\dim X$.
\begin{enumerate}
\item{
$d<\dim X$. For degree reasons, it suffices to require the orthogonality between $B^d_p$ and degree $d$ basis which precedes $(p,d)$ in the lexicographical order, namely $B^e_1,\cdots,B^e_{p-1}$. If we denote the Poincar\'e pairing by regular multiplication, then the condition can be written in matrix notation as
$$
\left(
\begin{array}{c}
\beta^d_{p,1}\\
\cdots\\
\beta^d_{p,k(p,d)}
\end{array}
\right)
\left(
\begin{array}{ccc}
\beta^e_{j,1},&\cdots,&\beta^e_{j,k(j,e)}
\end{array}
\right)=0,
$$

for $j=0,\cdots,\hat{q},\cdots,i-1$, and

$$
\left(
\begin{array}{c}
\beta^d_{p,1}\\
\cdots\\
\beta^d_{p,k(p,d)}
\end{array}
\right)
\left(
\begin{array}{ccc}
\beta^e_{q,1},&\cdots,&\beta^e_{q,k(q,e)}
\end{array}
\right)=I_{k(p,d)},
$$
where $I$ denotes the identity matrix. Pluging in \eqref{gram}, we have
 
$$
\left(
\begin{array}{ccc}
\widetilde{\beta^d_{p,1}}\\
\cdots\\
\widetilde{\beta^d_{p,k(p,d)}}
\end{array}
\right)
\left(
\begin{array}{ccc}
\beta^e_{j,1},&\cdots,&\beta^e_{j,k(j,e)}
\end{array}
\right)
+
\sum_{i=q+1}^{p-1} A_i
\left(
\begin{array}{c}
\beta^d_{i,1}\\
\cdots\\
\beta^d_{i,k(i,d)}
\end{array}
\right)
\left(
\begin{array}{ccc}
\beta^e_{j,1},&\cdots,&\beta^e_{j,k(j,e)}
\end{array}
\right)=0
$$
for $j=0,\cdots,\hat{q},\cdots,p-1$, and
$$
\left(
\begin{array}{ccc}
\widetilde{\beta^d_{p,1}}\\
\cdots\\
\widetilde{\beta^d_{p,k(p,d)}}
\end{array}
\right)
\left(
\begin{array}{ccc}
\beta^e_{q,1},&\cdots,&\beta^e_{q,k(q,e)}
\end{array}
\right)
+
\sum_{i=q+1}^{p-1} A_i
\left(
\begin{array}{c}
\beta^d_{i,1}\\
\cdots\\
\beta^d_{i,k(i,d)}
\end{array}
\right)
\left(
\begin{array}{ccc}
\beta^e_{q,1},&\cdots,&\beta^e_{q,k(q,e)}
\end{array}
\right)=I.
$$

\noindent
The second condition is always satisfied by $q+i<2r(f)$ and by Lemma \ref{zero}. The first condition is true when $j<q$ for the same reason. When $q\le j\le p-1$, by induction hypothesis, the first condition is reduced to

$$
\left(
\begin{array}{ccc}
\widetilde{\beta^d_{p,1}}\\
\cdots\\
\widetilde{\beta^d_{p,k(p,d)}}
\end{array}
\right)
\left(
\begin{array}{ccc}
\beta^e_{j,1},&\cdots,&\beta^e_{j,k(j,e)}
\end{array}
\right)
+
A_{2r(f)-j}=0.
$$
This solves $A_{2r(f)-j}$. Note that $q+1\le j\le p-1$, so $q+1\le 2r(f)-j\le p-1$ (note that $p+q=2r(f)$). That means that all $A_i$ are determined.
}

\item{
$j>\dim X$. The only difference in this case is that $B^e_p$ is already done, so we need one more condition to require $B^d_p$ to be orthogonal to $B^e_p$. To make this work, the sum taken in \eqref{gram} need to be from $q$ to $p-1$ instead of from $q+1$ to $p-1$. The computation is similar. 
}
\item{
$j=\dim X$. In this case $\widetilde{B^{\dim X}_p}$ need to be modified to be orthogonal to itself. The condition to be satisfied is exactly the same as $j>\dim X$ case, the result is slightly different: the matrix $A_q$ is different by a factor $2$.
}
\end{enumerate}
This completes the induction. In particular, the dual basis $\alpha^d_{p,i}=\pm\beta^e_{q,i}$, so $\mathfrak{p}(\alpha^d_{p,i})+\mathfrak{p}(\beta^e_{q,i})=2r(f)$.
\end{proof}

\begin{remark}
The assumption that $X$ and $Y$ are smooth varieties is not necessary. In fact the construction works for the intersection cohomology for singular varieties. 
\end{remark}

\begin{remark}
We point out an easy but important fact about $B$. The basis $B$ is filtered in the sense that 
$$
P_pH^*(X)=\text{Span }\{\beta\in B\mid \mathfrak{p}(\beta)\le p\}.
$$
\end{remark}

\noindent
By the additivity of perversities with respect to tensor products, we have the following.
\begin{corollary}\label{basis^n}
Let $f:X\to Y$ be a morphism between smooth projective varieties. Let $B=\{\beta_1,\cdots,\beta_k\}$ be the basis of $H^*(X)$ in Proposition \ref{basis}. Then the set $B^n$ defined by 
$$
B^n:=\{\beta_{i_1}\otimes\cdots\otimes\beta_{i_n}\mid 1\le i_1,\cdots, i_n\le k\}
$$
is a basis of $H^*(X^n)$. Furthermore, this basis is filtered with respect to the perverse filtration induced by map $f^n:X^n\to Y^n$ .
\end{corollary}

\subsection{Perversity estimation of small diagonals}
\noindent
In this section, we study the perversity of the small diagonal of the cartesian self-product。 This estimation is crucial to prove the multiplicativity of perverse filtration of Hilbert schemes.

\begin{proposition}\label{diagonal}
 Let $f:X\to Y$ be any morphism between smooth projective varieties. The small diagonal embedding $\Delta_n:X\to X^n$ induces a Gysin push-forward of cohomology
$$
\Delta_{n,*}:H^*(X)\to H^{*+2(n-1)\dim X}(X).
$$
Suppose the perverse filtration for $f:X\to Y$ is multiplicative, i.e. for any two classes $\alpha_1,\alpha_2\in H^*(X)$, we have 
$\mathfrak{p}(\alpha_1\cup\alpha_2)\le \mathfrak{p}(\alpha_1)+\mathfrak{p}(\alpha_2)$. Then for any $\gamma\in H^*(X)$, we have that
$$
\mathfrak{p}(\Delta_{n,*}(\gamma))\le \mathfrak{p}(\gamma)+2(n-1)r(f),
$$ 
where the perversity on the left side is defined by the map $f^n:X^n\to Y^n$ and the one on the right side is defined by $f:X\to Y$.
\end{proposition}

\noindent
We need an easy fact to prove the proposition.

\begin{lemma}\label{dual}
Let $X$ be a compact smooth manifold. Let $\beta_1,\cdots,\beta_k$ be an additive $\Q$-basis of $H^*(X)$. Let $\alpha_1,\cdots,\alpha_k$ be the dual basis with respect to the Poincar\'e pairing, namely $\langle\alpha_i,\beta_j\rangle=\delta_{ij}$. Then
$$
\Delta_{2,*}(\gamma)=\sum_{i=1}^k{\alpha_i\otimes(\beta_i\cup\gamma)}.
$$
\end{lemma}

\begin{proof}
Let $\pr_1,\pr_2:X\times X\to X$ be the projection maps to the two factors. Any cohomology class $\Phi\in H^*(X\times X)$ induces a correspondence 
\begin{eqnarray*}
[\Phi]:H^*(X)&\to& H^*(X)\\
\xi&\mapsto&\pr_{2,*}(\pr_1^*(\xi)\cup\Phi).
\end{eqnarray*}
Now the correspondence induced by left hand side is
\begin{equation*}
\begin{array}{rrl}
\displaystyle\left[\Delta_{2,*}(\gamma)\right](\xi)&=&\pr_{2,*}(\pr_1^*(\xi)\cup\Delta_{2,*}(\gamma))\\
&=&\pr_{2,*}(\xi\otimes 1\cup\Delta_{2,*}(\gamma))\\
&=&\pr_{2,*}\Delta_{2,*}(\Delta_2^*(\xi\otimes1)\cup\gamma)\\
&=&\Delta_2^*(\xi\otimes1)\cup\gamma\\
&=&\xi\cup\gamma,
\end{array}
\end{equation*}
where the second equality is due to the projection formula, and the third equality uses $\Delta_2\circ\pr_2=\text{id}$. The correspondence on right hand side computes as
\begin{equation*}
\begin{array}{rrl}
\displaystyle\left[\sum_{i=1}^k{\alpha_i\otimes\beta_i\cup\gamma}\right](\beta_j)&=&\displaystyle\sum_{i=1}^k{\pr_{2,*}(\beta_j\cup\alpha_i\otimes\beta_i\cup\gamma)}\\
&=&\pr_{2,*}(\beta_j\cup\alpha_j\otimes\beta_j\cup\gamma)\\
&=&\beta_j\cup\gamma.
\end{array}
\end{equation*}
Here we use the fact that the nontrivial push-forward takes place only when $\beta_j\cup\alpha_i$ is a cohomology class of top degree, and hence in this case $\beta_j\cup\alpha_i=\langle\beta_j,\alpha_i\rangle=\delta_{ij}$ by our choice of $\{\alpha_i\}$ and $\{\beta_i\}$. So there is only one non-zero pairing left in the summation. Now extending by linearity, we have
$$
\left[\sum_{i=1}^k{\alpha_i\otimes\beta_i\cup\gamma}\right](\xi)=\xi\cup\gamma.
$$
So the lemma follows.
\end{proof}

\begin{proof}[Proof of Proposition \ref{diagonal}]
We use induction on $n$ to prove the statement. Since $n=1$ is trivial, we prove for $n=2$ as induction basis.\\
Let $\{\beta^d_{p,i}\}$, $\{\alpha^d_{p,i}\}$ be the basis in Proposition \ref{basis}. Then by Lemma \ref{dual} we have
$$
\Delta_{2,*}(\gamma)=\sum_{p,d,i}{\alpha^d_{p,i}\otimes(\beta^d_{p,i}\cup\gamma)}
$$
Now by Proposition \ref{product}, Proposition \ref{basis}, and the hypothesis that the perverse filtration for $f:X\to Y$ is multiplicative, we have
\begin{equation*}
\begin{array}{rrl}
 \mathfrak{p}(\Delta_{2,*}(\alpha))&\le& \max_{p,d,i}\mathfrak{p}(\alpha^d_{p,i}\otimes(\beta^d_{p,i}\cup\gamma))\\
&\le& \max_{p,d,i}(\mathfrak{p}(\alpha^d_{p,i})+\mathfrak{p}(\beta^d_{p,i})+\mathfrak{p}(\gamma))\\
&\le& 2r(f)+\mathfrak{p}(\gamma).
\end{array}
\end{equation*}
\noindent
For general $n$, $\Delta_n$ can be decomposed into the following two diagonal maps. 

$$
X\xrightarrow{\Delta_{n-1}}X^{n-1}\xrightarrow{\Delta_2\times \text{Id}^{n-2}}X^n.
$$
\noindent
Then by induction hypothesis, we have

\begin{equation*}
\begin{array}{rrl}
\mathfrak{p}(\Delta_{n,*}(\gamma))&\le&\mathfrak{p}(\Delta_{n-1,*}\gamma)+2r(f)\\
&\le&\mathfrak{p}(\gamma)+2(n-2)r(f)+2r(f)\\
&=&\mathfrak{p}(\gamma)+2(n-1)r(f).
\end{array}
\end{equation*}
\end{proof}

\section{Hilbert scheme of points on surfaces}
\noindent
In this section we produce a perverse decomposition for the Hibert schemes of points on smooth surfaces in terms of a perverse decomposition for the fibered surface. We use Lehn's description of the ring structure of the cohomology of Hilbert schemes and the perversity estimation of the diagonal in self-cartesian product to prove the multiplicativity of the perverse filtration for Hilbert schemes. 

\subsection{Ring structure of Hilbert scheme of K3 surfaces}

In this section we recall the notation, definition and results on the cup product on the cohomology ring of the Hilbert scheme of points on surface with numerical trivial canonical bundle. Let $S$ be a projective surface with numerically trivial canonical bundle. Let $A=H^*(S;\Q)$ be the cohomology with $\Q$ coefficients. Let $[n]$ denote the set $\{1,\cdots,n\}$.

\begin{definition}[\cite{lehn} 2.1]
Let $I$ be a finite set of cardinality $n$. Define 
$$
A^I=\left(\bigoplus_{f:[n]\xrightarrow{\sim} I}A_{f(1)}\otimes\cdots\otimes A_{f(n)}\right)\bigg/ \mathfrak{S}_n.
$$
\end{definition}

\begin{remark}
In fact, $A^I$ is isomorphic to $A^{|I|}$. This isomorphism is canonical once an isomorphism $\varphi:[n]\to I$ is fixed. 
\end{remark}

\begin{definition}[\cite{lehn} 2.1]\label{4.3}
Let $\varphi:I\to J$ be a surjective map between sets. Then $\varphi$ induces a morphism $\varphi:S^J\to S^I$ by sending $(x_1,\cdots,x_{|J|})$ to $(x_{\varphi(1)},\cdots,x_{\varphi(|I|)})$. Define $\varphi_*$ and $\varphi^*$ to be the push-forward and pull-back map associated with $\varphi$ between the cohomology groups $H^*(S^I)$ and $H^*(S^J)$.
\end{definition}

\begin{remark}
The pull-back map can be described explicitly as follows. First note that $\varphi$ is a product of diagonal embedding map: the $j$-th copy of $S$ in $S^J$ is embedded diagonally in $S^{f^{-1}(j)}$. Pulling-back along the diagonal embedding is exactly the definition of cup product. So if we fix isomorphism $f:[n]\xrightarrow{\sim} I$, $g:[m]\xrightarrow{\sim} J$, we will have 
$$
\tilde{\varphi}:[n]\to I\to J \to [m]
$$
Therefore
\begin{eqnarray*}
\tilde{\varphi}^*:A^n&\to&A^m\\
a_1\otimes\cdots\otimes a_n&\mapsto&\bigotimes_{j=1}^m{\bigcup_{i\in\tilde{\varphi}^{-1}(j)}{a_i}}
\end{eqnarray*}
and
$$
\varphi^*:A^I\to A^n\xrightarrow{\tilde{\varphi}^*}A^m\to A^J
$$
It is easy to check that this is independent of choice of $f$ and $g$.
\end{remark}

\noindent
Now we define the wreath product of $A$ and $\mathfrak{S}_n$, which is used to describe the cohomology of Hilbert scheme of points on smooth surfaces. For a  permutation $\sigma\in \mathfrak{S}_n$ and a partition $\nu=1^{a_1}\cdots n^{a_n}$ of $n$, we say $\sigma$ is of type $\nu$ if $\sigma$ has exactly $a_i$ $i$-cycles. For $K$ a subgroup of $\mathfrak{S}_n$, and for a $K$-stable subset $E\subset[n]$, let $K\backslash E$ denote the set of orbits for the induced action of $K$ on $E$.

\begin{definition}[\cite{lehn} Lemma 2.7]
For $\sigma,\tau\in \mathfrak{S}_n$, the graph defect $g(\sigma,\tau):\langle\sigma,\tau\rangle\backslash[n]\to\Q$ is defined by
$$
g(\sigma,\tau)(E)=\frac{1}{2}(|E|+2-|\langle\sigma\rangle\backslash E|-|\langle\tau\rangle\backslash E|-|\langle\sigma\tau\rangle\backslash E|).
$$
In fact, $g(\sigma,\tau)$ takes value in non-negative integers.
\end{definition}

\begin{definition}[\cite{lehn} 2.8]\label{4.6}
The wreath product of $A$ and symmetric group $\mathfrak{S}_n$ as follows.
$$
A\{\mathfrak{S}_n\}:=\bigoplus_{\sigma\in \mathfrak{S}_n}A^{\otimes \langle\sigma\rangle\backslash[n]}[-2|\sigma|]\cdotp\sigma.
$$
$\mathfrak{S}_n$ acts on $A\{\mathfrak{S}_n\}$ as follows: the action of $\tau\in \mathfrak{S}_n$ on $[n]$ induces a bijection
\begin{eqnarray*}
\sigma:\langle\sigma\rangle\backslash[n]&\to&\langle\tau\sigma\tau^{-1}\rangle\backslash[n]\\
x&\mapsto&\tau x
\end{eqnarray*}
for each $\sigma$ and hence an isomorphism
\begin{eqnarray*}
\tilde{\tau}:A\{\mathfrak{S}_n\}&\to&A\{\mathfrak{S}_n\}\\
a\sigma&\mapsto&\tau^*(a)\tau\sigma\tau^{-1}.
\end{eqnarray*}
Let
$$
A^{[n]}:=(A\{\mathfrak{S}_n\})^{\mathfrak{S}_n}
$$
be the subspace of invariants.
\end{definition}

\noindent
Any inclusion $H\subset K$ of subgroups of $\mathfrak{S}_n$ induces a surjection $H\backslash[n]\twoheadrightarrow K\backslash[n]$ of set of orbits and hence induces a pull-back map
$$
f^{H,K}:A^{\otimes H\backslash[n]}\to A^{\otimes K\backslash[n]}
$$
and a push-forward map
$$
f_{K,H}:A^{\otimes K\backslash[n]}\to A^{\otimes H\backslash[n]}.
$$

\begin{definition}[\cite{lehn} 2.12]
For $\sigma,\tau\in \mathfrak{S}_n$, define
\begin{eqnarray*}
m_{\sigma,\tau}:A^{\otimes \langle\sigma\rangle\backslash[n]}\otimes A^{\otimes \langle\tau\rangle\backslash[n]} &\to&
A^{\otimes \langle\sigma\tau\rangle\backslash[n]}\\
a\otimes b&\mapsto& f_{\langle\sigma,\tau\rangle,\langle\sigma\tau\rangle}(f^{\langle\sigma\rangle,\langle\sigma,\tau\rangle}(a)\cdotp f^{\langle\tau\rangle,\langle\sigma,\tau\rangle}(b)\cdotp e^{g(\sigma,\tau)})
\end{eqnarray*}
where $e$ is the Euler class of $S$.
\end{definition}

\begin{proposition}[\cite{lehn} Proposition 2.13]
The product $A\{\mathfrak{S}_n\}\times A\{\mathfrak{S}_n\}\xrightarrow{\cdotp}A\{\mathfrak{S}_n\}$ defined by
$$
a\sigma\cdotp b\tau:=m_{\sigma,\tau}(a\otimes b)\sigma\tau
$$
is associative and $\mathfrak{S}_n$-equivariant. So it descends to a product on $A^{[n]}$.
\end{proposition}

\begin{theorem}[\cite{lehn} Theorem 3.2]\label{lehn}
Let $S$ be a smooth projective surface with numerically trivial canonical divisor.  Let $S^{[n]}$ denote the Hilbert scheme of $n$ points on  the surface $S$. Then there is a canonical isomorphism of graded rings
$$
H^*(S;\Q)^{[n]}\xrightarrow{\cong}H^*(S^{[n]};\Q).
$$
\end{theorem}

\noindent
For later use, we need a non-compact version of this theorem. 

\begin{proposition}\label{4.10}
Let $S$ be a smooth quasi-projective surface with trivial canonical divisor. Suppose $\overline{S}$ is a smooth projective surface which contains $S$ as an open set with the property that the natural restriction map $H^*(\overline{S})\to H^*(S)$ is surjective. Then there is a canonical isomorphism of graded rings
$$
H^*(S;\Q)^{[n]}\xrightarrow{\cong}H^*(S^{[n]};\Q).
$$  
\end{proposition}

\begin{proof}
First, we notice that Theorem \ref{lehn} is obtained by setting canonical bundle $K=0$ in the main theorem of \cite{lehn0}.
In our cases, by G\"ottsche formula, $H^*(\overline{S})\to H^*(S)$ being surjective implies $H^*\left(\overline{S}^{[n]}\right)\to H^*\left(S^{[n]}\right)$ is surjective. Therefore, the cup product in $H^*\left(S^{[n]}\right)$ is completely determined by the ring structure of $H^*\left(\overline{S}^{[n]}\right)$. Since $\overline{S}$ is projective but may not have trivial canonical bundle, we have to use the ring structure described implicitly in \cite{lehn0}. However, after resticting to $S^{[n]}$, all terms involving canonical bundle $K_{\overline{S}}$ vanish, which is exactly the same situation as the proof in Theorem \ref{lehn}. So the theorem also works for non-compact surface with trivial canonical bundle.
\end{proof}

\subsection{Perverse filtration for Hilbert schemes of fibered surfaces}
To descrie the perverse filtration for Hilbert schemes of fibered surfaces, we need to introduce some notation first. Partitions of $n$ are denoted as $\nu=1^{a_1}\cdots n^{a_n}$, where $\sum_{i=1}^n{ia_i}=n$. The length of a partition is denoted by $l(\nu)=\sum_{i=1}^n{a_i}$. Set $\mathfrak{S}_\nu:=\mathfrak{S}_{a_1}\times\cdots\times \mathfrak{S}_{a_n}$. For a quasi-projective vareity $X$, set $X^{(\nu)}:=X^{l(\nu)}/\mathfrak{S}_\nu=X^{(a_1)}\times\cdots\times X^{(a_n)}$. For $K\in D^b_c(X)$, denote the multi-symmetric and multi-alternating external tensor product by
$$
K^{(\nu)}=\boxtimes_{i=1}^n K^{(a_i)}\in D^b_c(X^{(\nu)}),
$$ 
$$
K^{\{\nu\}}=\boxtimes_{i=1}^n K^{\{a_i\}}\in D^b_c(X^{(\nu)}).
$$
We still have 
$$
(K[a])^{(\nu)}=
\begin{cases}
K^{(\nu)}[l(\nu)a], & a\text{ is even,}\\
K^{\{\nu\}}[l(\nu)a], & a\text{ is odd.}
\end{cases}
$$
In fact, since the external tensor product is compatible with push-forward and perversity, all result in Section 2.3 can be generalized to the multi-symmetric or multi-alternating context.\\

\noindent
Let $f:S\to C$ be a proper map from a smooth quasi-projective surface to a smooth quasi-projective curve. Then we have the following diagram.
$$
\begin{tikzcd}
\ &\ & S^{[n]}
 \arrow{d}[swap]{\pi}
 \arrow[bend left]{dd}{h}\\
S^{l(\nu)}\ar{r}{/\mathfrak{S}_\nu}\ar{d}[swap]{f^{l(\nu)}}&
S^{(\nu)}\ar{r}{r_S^{(\nu)}}\ar{d}[swap]{f^{(\nu)}}&
S^{(n)}\ar{d}[swap]{f^{(n)}}\\
C^{l(\nu)}\ar{r}{/\mathfrak{S}_\nu}&
C^{(\nu)}\ar{r}{r_C^{(\nu)}}&C^{(n)}\\
\end{tikzcd}
$$

\noindent
To obtain a perverse decomposition for the map $h$, we need the following result:

\begin{theorem}[\cite{hilb} Theorem 4.1.1] \label{hilb-chow}
 Let $S$ be a smooth quasi-projective algebraic surface. Then we have the decomposition theorem for the Hilbert-Chow morphism.
$$
R\pi_*\Q_{S^{[n]}}[2n]\cong\bigoplus_\nu Rr_{S,*}^{(\nu)}\Q_{S^{(\nu)}}[2l(\nu)].
$$
\end{theorem}

\noindent
The main result of this section is the following.

\begin{proposition}\label{4.4}
Let $f:S\to C$ be a proper map from a smooth quasi-projective surface to a smooth quasi-projective curve. Let
\begin{equation*}
Rf_*\Q_S[1]=\mathcal{P}_0\oplus \mathcal{P}_1[-1]\oplus \mathcal{P}_2[-2]
\end{equation*}
 be a perverse decomposition, where $\mathcal{P}_0,\mathcal{P}_1,\mathcal{P}_2$ are perverse sheaves on $C$. Then a perverse decomposition of the morphism $h:S^{[n]}\to C^{(n)}$ is given by:
\begin{equation*}
\begin{split}
Rh_*\Q_{S^{[n]}}[n]\cong&\bigoplus_{\nu=1^{a_1}\cdots n^{a_n}} Rr_{C,*}^{(\nu)}\left((\mathcal{P}_0[0]\oplus \mathcal{P}_1[-1]\oplus \mathcal{P}_2[-2])^{\{\nu\}}\right)[l(\nu)-n]\\
\cong&\bigoplus_{\nu=1^{a_1}\cdots n^{a_n}} Rr_{C,*}^{(\nu)}\left(\bigoplus_{\bold{r}+\bold{s}+\bold{t}=\bold{a}}Rq_{\bold{r},\bold{s},\bold{t},*}\mathcal{P}_0^{\{\bold{r}\}}\boxtimes \mathcal{P}_1^{(\bold{s})}\boxtimes \mathcal{P}_2^{\{\bold{t}\}}\right)[-C(\bold{r},\bold{s},\bold{t})]
\end{split}
\end{equation*}
where $C(\bold{r},\bold{s},\bold{t})=n-l(\nu)+\sum_{i=1}^n(s_i+2t_i)$, $\bold{r}=(r_1,\cdots,r_n)$, and similarly for $\bold{s}$ and $\bold{t}$. Here $\bold{r}+\bold{s}+\bold{t}=\bold{a}$ means that $r_i+s_i+t_i=a_i$ holds for any $i$, and 
$$
q_{\bold{r},\bold{s},\bold{t}}:\prod_{i=1}^{l(\nu)}C^{(r_i)}\times C^{(s_i)}\times C^{(t_i)}\to\prod_{i=1}^{l(\nu)}C^{(a_i)}
$$
 is induced by the natural map $C^{(r_i)}\times C^{(s_i)}\times C^{(t_i)}\to C^{(a_i)}$.
\end{proposition}

\begin{proof}
The proof is formal.
\begin{equation*}
\begin{array}{rl}
Rh_*\Q_{S^{[n]}}[n]\cong& Rf_*^{(n)}R\pi_*\Q_{S^{[n]}}[n]\\
\cong&\displaystyle Rf_*^{(n)}\left(\bigoplus_\nu Rr_{S,*}^{(\nu)}\Q_{S^{(\nu)}}[2l(\nu)-n]\right)\\
\cong&\displaystyle \bigoplus_\nu Rf_*^{(n)}Rr_{S,*}^{(\nu)}\Q_{S^{(\nu)}}[2l(\nu)-n]\\
\cong&\displaystyle \bigoplus_\nu Rr_{C,*}^{(\nu)}Rf_*^{(\nu)}\Q_{S^{(\nu)}}[2l(\nu)-n]\\
\cong&\displaystyle \bigoplus_\nu Rr_{C,*}^{(\nu)}\left(Rf_*\Q_S\right)^{(\nu)}[2l(\nu)-n]\\
\cong&\displaystyle \bigoplus_\nu Rr_{C,*}^{(\nu)}\left(Rf_*\Q_S[1]\right)^{\{\nu\}}[l(\nu)-n]\\
\cong&\displaystyle \bigoplus_\nu Rr_{C,*}^{(\nu)}\left(\mathcal{P}_0[0]\oplus\mathcal{P}_1[-1]\oplus\mathcal{P}_2[-2]\right)^{\{\nu\}}[l(\nu)-n]\\
\cong&\displaystyle \bigoplus_{\nu=1^{a_1}\cdots n^{a_n}} Rr_{C,*}^{(\nu)}\left(\bigoplus_{\bold{r}+\bold{s}+\bold{t}=\bold{a}}Rq_{\bold{r},\bold{s},\bold{t},*}\mathcal{P}_0^{\{\bold{r}\}}\boxtimes \mathcal{P}_1^{(\bold{s})}\boxtimes \mathcal{P}_2^{\{\bold{t}\}}\right)[-C(\bold{r},\bold{s},\bold{t})].
\end{array}
\end{equation*}

\noindent
Here we use Lemma \ref{5.3} and Proposition \ref{5.4} and the fact that $\dim S-r(h)=1$ is odd. By Proposition \ref{product} and Proposition \ref{symmetric} and the fact that the map $q_{\bold{r},\bold{s},\bold{t}}$ is finite, all terms in the parenthesis are perverse. Note that $r_{C,*}^{(\nu)}$ is a closed embedding, so $Rr_{C,*}^{(\nu)}$ is $t$-exact, it preserves perversity. Therefore the above gives a perverse decomposition.
\end{proof}

\begin{remark}
In the Proposition \ref{4.4}, we have symmetric product for odd perversity term and alternating products for even perversity terms. This counter-intuitive result is due to the fact that $\mathcal{P}_i[-i]$ are direct summands of $Rf_*\Q_S[1]$ rather than $Rf_*\Q_S$. See Proposition \ref{5.4}.
\end{remark}

\begin{corollary}\label{perversity}
Under the isomorphism 
\begin{equation*}
H^*\left(S^{[n]}\right)=\bigoplus_\nu\left(H^*(S^{l(\nu)})\right)^{\mathfrak{S}_\nu}[2l(\nu)-2n],
\end{equation*}
the perverse filtration can be identified as 
$$
P_pH^*\left(S^{[n]}\right)=\bigoplus_\nu\left(P_{p+l(\nu)-n} H^*(S^{l(\nu)})\right)^{\mathfrak{S}_\nu}[2l(\nu)-2n],
$$ 
where the perversity on the right side is taken with respect to $f^{l(\nu)}:S^{l(\nu)}\to C^{l(\nu)}.$
\end{corollary}

\begin{proof}
By $t$-exactness of $Rr^{(\nu)}_{C,*}$ and Proposition \ref{4.4}, we have
$$
\begin{array}{rrl}
{^{\mathfrak{p}}\tau_{\le p}}Rh_*\Q_{S^{[n]}}[n]&=&\displaystyle \bigoplus_{\nu=1^{a_1}\cdots n^{a_n}}Rr_{C,*}^{(\nu)}\left( {^{\mathfrak{p}}\tau_{\le p+l(\nu)-n}}(Rf_*\Q_S[1])^{\{\nu\}}  \right)[l(\nu)-n]\\
&=&\displaystyle \bigoplus_{\nu=1^{a_1}\cdots n^{a_n}}Rr_{C,*}^{(\nu)}\left( Rq_*{^{\mathfrak{p}}\tau_{\le p+l(\nu)-n}}(Rf_*\Q_S[1])^{\boxtimes l(\nu)}  \right)^{\text{sign-}\mathfrak{S}_n}[l(\nu)-n].\\
\end{array}
$$
where $q$ is the quotient map denoted as $/\mathfrak{S}_\nu$ in the previous diagram and the last isomorphism is due to Corollary \ref{5.5}. After taking the cohomology, we have
$$
\begin{array}{rrl}
P_pH^*(S^{[n]})[n]
&=&\mathbb H\left(C^{(n)}, {^{\mathfrak{p}}\tau_{\le p}}Rh_*\Q_{S^{[n]}}[n]\right)\\
&=&\displaystyle \bigoplus_{\nu=1^{a_1}\cdots n^{a_n}}\mathbb{H}\left(C^{(\nu)}, \left(Rq_*{^{\mathfrak{p}}\tau_{\le p+l(\nu)-n}}(Rf_*\Q_S[1])^{\boxtimes l(\nu)}  \right)^{\text{sign-}\mathfrak{S}_n}\right)[l(\nu)-n]\\
&=&\displaystyle \bigoplus_{\nu=1^{a_1}\cdots n^{a_n}}\mathbb{H}\left(C^{l(\nu)}, {^{\mathfrak{p}}\tau_{\le p+l(\nu)-n}}(Rf_*\Q_S[1])^{\boxtimes l(\nu)}  \right)^{\text{sign-}\mathfrak{S}_n}[l(\nu)-n]\\
&=&\displaystyle \bigoplus_{\nu=1^{a_1}\cdots n^{a_n}}\left(P_{p+l(\nu)-n}H^*\left(C^{l(\nu)},(Rf_*\Q_S[1])^{\boxtimes l(\nu)}\right)\right)^{\text{sign-}\mathfrak{S}_n}[l(\nu)-n]\\
&=&\displaystyle \bigoplus_{\nu=1^{a_1}\cdots n^{a_n}}\left(P_{p+l(\nu)-n}H^*\left(S^{l(\nu)},(\Q_S[1])^{\boxtimes l(\nu)}\right)\right)^{\text{sign-}\mathfrak{S}_n}[l(\nu)-n]\\
&=&\displaystyle \bigoplus_{\nu=1^{a_1}\cdots n^{a_n}}\left(P_{p+l(\nu)-n}H^*\left(S^{l(\nu)},\Q_S^{\boxtimes l(\nu)}\right)\right)^{\mathfrak{S}_n}[2l(\nu)-n]\\
&=&\displaystyle \bigoplus_{\nu=1^{a_1}\cdots n^{a_n}}\left(P_{p+l(\nu)-n}H^*(S^{l(\nu)})\right)^{\mathfrak{S}_n}[2l(\nu)-n].\\
\end{array}
$$
So the result follows.
\end{proof}

\noindent
In fact, we may define an abstract perversity function on the wreath product $H^*(S)\{\mathfrak{S}_n\}$ (Definition \ref{4.6}) which is easier to handle with. We will show that after restrict to the $\mathfrak{S}_n$-invariant part, it is the same as the one defined by the map $S^{[n]}\to C^{(n)}$.

\begin{definition}
Let $f:S\to C$ be a morphism from a smooth quasi-projective surface to a smooth quasi-projective curve. Let $\mathfrak{B}$ be any basis of $H^*(S)$ which is filtered with respect to the perverse filtration for the map $f:S\to C$ . Then we have that
$$
\mathfrak{B}\{\mathfrak{S}_n\}:=\left\{\bigotimes_{i=1}^n\bigotimes_{j=1}^{a_i}\alpha_{ij}\cdotp\sigma\mid \alpha_{ij}\in \mathfrak{B},\sigma\text{ is of type }1^{a_1}\cdots n^{a_n}\right\}
$$
is a basis of $H^*(S)\{\mathfrak{S}_n\}$. We define an abstract perversity on $\mathfrak{B}\{\mathfrak{S}_n\}$ as
$$
\mathfrak{p}_{abs}\left(
\bigotimes_{i=1}^n\bigotimes_{j=1}^{a_i}\alpha_{ij}\cdotp\sigma
\right)=\sum_{i=1}^n\sum_{j=1}^{a_i}{\mathfrak{p}(\alpha_{ij})}+\sum_{i=1}^n{(i-1)a_i},
$$
and extend by linearity in the sense that 
$$
P_pH^*(S)\{\mathfrak{S}_n\}=\text{Span}\{\beta\in \mathfrak{B}\{\mathfrak{S}_n\}\mid \mathfrak{p}_{abs}(\beta)\le p\}.
$$
In particular, the basis is filtered with respect to the abstract perverse filtration by definition.
\end{definition}

\begin{proposition} \label{4.15}
The abstract perversity is invariant under the $\mathfrak{S}_n$-action. Furthermore, after restriction to $H^*(S^{[n]})$, it is the same as the perversity given by the morphism $h:S^{[n]}\to C^{(n)}$ in Corollary \ref{perversity}.
\end{proposition}

\begin{proof}
Note that $\mathfrak{S}_n$ acts on cohomology by permuting the factors, so the abstract perversity is invariant under $\mathfrak{S}_n$-action. 
By definition, we have 
$$
n-l(\nu)=\sum_{i=1}^n(i-1)a_i
$$ 
\noindent
and
$$
\mathfrak{p}_{abs}\left(\bigotimes_{i=1}^n\bigotimes_{j=1}^{a_i}\alpha_{ij}\cdotp\sigma
\right)=p\text{ if and only if }
\mathfrak{p}_{abs}\left(\bigotimes_{i=1}^n\bigotimes_{j=1}^{a_i}\alpha_{ij}\right)=p+n-l(\nu).
$$
where $\nu=1^{a_1}\cdots n^{a_n}$ be any partition of $n$. Comparing with Corollary \ref{perversity}, the abstract perversity is the same as the geometric perversity induced by the morphism $h$ on the basis. Note that on both sides the bases are filtered with respect to perverse filtration, so the abstract perverse filtration coincides with the geometric perverse filtration.
\end{proof}

\subsection{Multiplicativity of the perverse filtration}
\begin{proposition}\label{multiplicative}
Let $f:S\to C$ be a proper surjective morphism from a smooth quasi-projective surface to a smooth quasi-projective curve. Then the perverse filtration on $H^*(S)$ is multiplicative.
\end{proposition}

\begin{proof}
\noindent
The perverse decomposition of a proper map from a surface to curve can be found in \cite{surface} Theorem 3.2.3. Let $f:S\to C$ be a proper surjective map from a smooth surface to a smooth curve. Let $\hat{f}:\hat{S}\to\hat{C}$ be its smooth part. Let $j:\hat{C}\to C$ be the open embedding. Let $\hat{R}^i=R^i\hat{f}_*\Q_{\hat{S}}$. Then one has a non-canonical perverse decomposition of map $f$:
$$
 Rf_*\Q_S[1]\cong\left\{j_*\hat{R}^0[1]\right\}\bigoplus\left\{j_*\hat{R}^1[1]\oplus \oplus_{p\in C\setminus \hat{C}}\Q_p^{n_p-1}\right\}[-1]\bigoplus\{j_*\hat{R}^2[1]\}[-2]
$$
where $n_p$ is the number of irreducible components of fiber over $p$. Note that the ordinary Leray filtration is multiplicaive, and the only difference between ordinary Leray filtration and the perverse filtration is that the classes corresponding to $\oplus_{p\in C\setminus\hat{C}}\Q_p^{n_p-1}$ are shifted from $R^2f_*\Q_S[-2]$ to $\mathcal{P}_1[-1]$. Therefore, it suffices to check cup products with these classes. Furthermore, these classes are in perversity $1$, the only possibility to violate the multiplicativity is that their cup with perversity $0$ classes have perversity $2$. However, they are fundamental classes of irreducible components of special fibers, and perversity $0$ classes are pull-backs of classes on the curve. They don't meet if the pull-back class is not the surface itself, and cupping with fundamental class of the surface is the identity map. So in both cases, the multiplicativity is preserved.
\end{proof}

\begin{theorem}\label{multiplicative^n}
Let $f:S\to C$ be a surjective morphism from a smooth projective surface with numerically trivial canonical bundle to a smooth projective curve. Then the perverse filtration of $H^*(S^{[n]};\Q)$ with respect to the morphism $h:S^{[n]}\to C^{(n)}$ is multiplicative, namely, we have
$$
P_pH^*(S^{[n]};\Q)\cup P_{p'}H^*(S^{[n]};\Q)\subset P_{p+p'}H^*(S^{[n]};\Q)
$$
\end{theorem}

\noindent
To prove this theorem, we need the following two lemmata. In fact, they are stated with slightly more general hypothesis so that we can also apply them in quasi-projective case in Chapter 5. We use the notation introduced in section 4.1.

\begin{lemma}\label{pull-back}
Let $f:S\to C$ be a surjective proper morphism from a smooth quasi-projective surface to a smooth quasi-projective curve. Then for any surjective map between sets $\varphi:I\twoheadrightarrow J$, the pullback map (Definition \ref{4.3}) $\varphi^*:H^*(S)^I\to H^*(S)^J$ does not increase perversity, where the perversity is defined for the map $S^{|I|}\to C^{|I|}$ and $S^{|J|}\to C^{|J|}$.
\end{lemma}

\begin{proof}
Let $\mathfrak{B}$ be any filtered basis of $H^*(S)$ with respect to the perverse filtration for the map $f:S\to C$. By Corollary \ref{basis^n}, the basis $\mathfrak{B}^{|I|}$ is filtered, so it suffices to compute the perversity of the pull-back of the elements in $\mathfrak{B}^{|I|}$. Pick an element $\alpha_1\otimes\cdots\otimes\alpha_{|I|}\in \mathfrak{B}^{|I|}$, then
$$
\varphi^*(\alpha_1\otimes\cdots\otimes\alpha_{|I|})=\bigotimes_{j=1}^{|J|}\bigcup_{i\in \varphi^{-1}(j)}\alpha_i
$$
Note that by Proposition \ref{product}
$$
\mathfrak{p}\left(\alpha_1\otimes\cdots\otimes\alpha_{|I|}\right)=\sum_{j=1}^{|I|}{\mathfrak{p}(\alpha_j)}
$$
where perversity on the left side is defined by map $S^{|I|}\to C^{|I|}$, and the perversity on the right side is defined by $S\to C$. Now by Proposition \ref{product} and Lemma \ref{multiplicative}, we have that
$$
\mathfrak{p}\left(\bigotimes_{j=1}^{|J|}\bigcup_{i\in \varphi^{-1}(j)}\alpha_i\right)\le\sum_{j=1}^{|J|}{\sum_{i\in \varphi^{-1}(j)}{\mathfrak{p}(\alpha_i)}}=\sum_{j=1}^{|I|}{\alpha_j},
$$
where the perversity on the left side is defined by the map $S^{|J|}\to C^{|J|}$.
\end{proof}

\begin{lemma}\label{push-forward}
Let $f:S\to C$ be a surjective proper morphism from a smooth quasi-projective surface to a smooth quasi-projective curve. Suppose that we have the perversity estimation of small diagonals 
$$
\mathfrak{p}(\Delta_{n,*}(\gamma))\le \mathfrak{p}(\gamma)+2(n-1)
$$
for any $\gamma\in H^*(S)$ and small diagonal embedding $\Delta_n:S\to S^n$. Then for any surjective map $\varphi:I\twoheadrightarrow J$, the push-forward map
$$
\varphi_*:H^*(S)^J\to H^*(S)^I
$$
increases the perversity at most by $2(|I|-|J|)$.
\end{lemma}

\begin{proof}
Again, it suffices to prove for a basis element $\alpha_1\otimes\cdots\otimes\alpha_{|J|}\in \mathfrak{B}^{|J|}$.
Let $b_j=|\varphi^{-1}(j)|$.
By definition,
$$
\varphi_*(\alpha_1\otimes\cdots\otimes\alpha_{|J|})=\pm\bigotimes_{j=1}^{|J|}\Delta_{b_j,*}(\alpha_j).
$$
By the hypothesis, we have
\begin{equation*}
\begin{array}{rrl}
\displaystyle\mathfrak{p}\left(\pm\bigotimes_{j=1}^{|J|}\Delta_{b_j,*}(\alpha_j)\right)&\le&\displaystyle\sum_{j=1}^{|J|}{\mathfrak{p}\left(\Delta_{b_j,*}(\alpha_j)\right)}\\
&=&\displaystyle\sum_{j=1}^{|J|}{\mathfrak{p}(\alpha_j)+2(b_i-1)}\\
&=&\displaystyle\sum_{j=1}^{|J|}{\mathfrak{p}(\alpha_j)}+2\sum_{j=1}^{|J|}{b_i}-2\sum_{j=1}^{|J|}{1}\\
&=&\displaystyle\sum_{j=1}^{|J|}{\mathfrak{p}(\alpha_j)}+2(|I|-|J|),
\end{array}
\end{equation*}
where the perversity on the left side is defined by the map $S^{|I|}\to C^{|I|}$, the perversities on the right side of the first line is defined by the map $S^{b_j}\to C^{b_j}$.
\end{proof}

\begin{proof}[Proof of Theorem \ref{multiplicative^n}]
By Proposition \ref{4.15}, it suffices to prove that the abstract perverse filtration defined on $H^*(S;\Q)\{\mathfrak{S}_n\}$ is multiplicative. Furthermore, it suffices to prove the result for our filtered basis, namely
\begin{equation*}
\begin{array}{rl}
 &\displaystyle \mathfrak{p}\left(\bigotimes_{i=1}^n\bigotimes_{j=1}^{a_i}\alpha_{ij}\cdotp\sigma\cup\bigotimes_{i=1}^n\bigotimes_{j=1}^{a_i'}\alpha'_{ij}\cdotp\tau\right)\\
\le&\displaystyle 
\sum_{i=1}^n{\sum_{j=1}^{a_i}{\mathfrak{p}(\alpha_{ij})}}+\sum_{i=1}^n{(i-1)a_i}\\
+&\displaystyle 
\sum_{i=1}^n{\sum_{j=1}^{a'_i}{\mathfrak{p}(\alpha'_{ij})}}+\sum_{i=1}^n{(i-1)a'_i},
\end{array}
\end{equation*}
where $\alpha_{ij}$ and $\alpha'_{ij}$ run over basis $B$ obtained in Proposition \ref{basis}. By Proposition \ref{basis}, Corollary \ref{basis^n} and Proposition \ref{diagonal}, the hypotheses of Lemma \ref{pull-back} and Lemma \ref{push-forward} are satisfied. Note that the cup product formula computes independently on each orbit of $\langle\sigma,\tau\rangle$-action on $[n]$ individually. Let $E$ be an orbit of the action $\langle\sigma,\tau\rangle$ on $[n]$, i.e. $|\langle\sigma,\tau\rangle\backslash E|=1$. The product is computed by
\begin{equation*}
\begin{array}{rrl}
A^{\otimes \langle\sigma\rangle\backslash E}\cdotp\sigma|_E\otimes A^{\otimes \langle\tau\rangle\backslash E}\cdotp\tau|_E &\to&
A^{\otimes \langle\sigma\tau\rangle\backslash E}\cdotp\sigma\tau|_E\\
a\cdotp\sigma|_E\otimes a'\cdotp\tau|_E&\mapsto& f_{\langle\sigma,\tau\rangle,\langle\sigma\tau\rangle}(f^{\langle\sigma\rangle,\langle\sigma,\tau\rangle}(a)\cdotp
f^{\langle\tau\rangle,\langle\sigma,\tau\rangle}(a')\cdotp e^{g(\sigma,\tau)})\cdotp\sigma\tau|_E
\end{array}
\end{equation*}
for every $E$. Note that the Euler class $e$ is of top degree, hence $e^g=0$ for $g\ge 2$, so that it suffices to consider the following two cases.
\begin{enumerate}
\item{
$g(\sigma,\tau)=0$. By Lemma \ref{pull-back} and Lemma \ref{push-forward}, we have
\begin{equation*}
\begin{array}{rrl}
& &\displaystyle\mathfrak{p}\left(f_{\langle\sigma,\tau\rangle,\langle\sigma\tau\rangle}(f^{\langle\sigma\rangle,\langle\sigma,\tau\rangle}(a)\cdotp
f^{\langle\tau\rangle,\langle\sigma,\tau\rangle}(a')\cdotp e^{g(\sigma,\tau)})\cdotp\sigma\tau|_E\right)\\
&=&\displaystyle\mathfrak{p}\left(f_{\langle\sigma,\tau\rangle,\langle\sigma\tau\rangle}(f^{\langle\sigma\rangle,\langle\sigma,\tau\rangle}(a)\cdotp
f^{\langle\tau\rangle,\langle\sigma,\tau\rangle}(a'))\right)+|E|-|\langle\sigma\tau\rangle\backslash E|\\
&=&\displaystyle\mathfrak{p}\left(f^{\langle\sigma\rangle,\langle\sigma,\tau\rangle}(a)\cdotp
f^{\langle\tau\rangle,\langle\sigma,\tau\rangle}(a'))\right)+2(|\langle\sigma\tau\rangle\backslash E|-1)+|E|-|\langle\sigma\tau\rangle\backslash E|\\
&=&\displaystyle\mathfrak{p}(a)+\mathfrak{p}(a')+|E|+|\langle\sigma\tau\rangle\backslash E|-2\\
&=&\displaystyle\mathfrak{p}(a\cdotp\sigma)-(|E|-|\langle\sigma\rangle\backslash E|)+\mathfrak{p}(a'\cdotp\tau)-(|E|-|\langle\tau\rangle\backslash E|)+|E|+|\langle\sigma\tau\rangle\backslash E|-2\\
&=&\displaystyle\mathfrak{p}(a\cdotp\sigma)+\mathfrak{p}(a'\cdotp\tau)-2g(\sigma,\tau)\\
&=&\displaystyle\mathfrak{p}(a\cdotp\sigma)+\mathfrak{p}(a'\cdotp\tau)
\end{array}
\end{equation*}
}
\item{
$g(\sigma,\tau)=1$. Since $e$ itself is already in top degree, so the only nonzero case arise for $a=a'=1$. Then
\begin{equation*}
\begin{array}{rrl}
& &\displaystyle\mathfrak{p}\left(f_{\langle\sigma,\tau\rangle,\langle\sigma\tau\rangle}(f^{\langle\sigma\rangle,\langle\sigma,\tau\rangle}(1)\cdotp
f^{\langle\tau\rangle,\langle\sigma,\tau\rangle}(1)\cdotp e^{g(\sigma,\tau)})\cdotp\sigma\tau|_E\right)\\
&=&\displaystyle\mathfrak{p}\left(f_{\langle\sigma,\tau\rangle,\langle\sigma\tau\rangle}(e)\right)+|E|-|\langle\sigma\tau\rangle\backslash E|\\
&=&\displaystyle2+2(|\langle\sigma\tau\rangle\backslash E|-1)+|E|-|\langle\sigma\tau\rangle\backslash E|\\
&=&\displaystyle|E|+|\langle\sigma\tau\rangle\backslash E|\\
&=&\displaystyle|E|-|\langle\sigma\rangle\backslash E|+|E|-|\langle\tau\rangle\backslash E|\\
&=&\displaystyle\mathfrak{p}(1\cdotp\sigma)+\mathfrak{p}(1\cdotp\tau)
\end{array}
\end{equation*}
The last but one equality is due to $g(\sigma,\tau)=1$, which means $|E|=|\langle\sigma\rangle\backslash E|+|\langle\tau\rangle\backslash E|+|\langle\sigma\tau\rangle\backslash E|$.
}
\end{enumerate}
\end{proof}

\noindent
An application of the theorem is the multiplicativity of perverse filtration for the elliptic fibration of Hilbert schemes of points on K3 surfaces.

\begin{theorem}
Let $S$ be an elliptic K3 surface and $f:S\to \P^1$ be the elliptic fibration. Then the perverse filtration on $H^*(S^{[n]})$ defined by the natural map $h:S^{[n]}\to \P^n$ is multiplicative. 
\end{theorem}

\subsection{The Hilbert schemes of a $P=W$ package}
\begin{definition}
A $P=W$ package is a $5$-tuple $(X_P,X_W,h,A,\Xi)$ where
\begin{enumerate}
\item{$X_P,X_W,A$ are smooth quasi-projective varieties. $h:X_P\to A$ is proper morphism. $\Xi:X_P\to X_W$ is a diffeomorphism.}
\item{$P_kH^*(X_P)=W_{2k}H^*(X_W)=W_{2k+1}H^*(X_W)$ for any $k$. Here perverse filtration is defined for map $f$, and the identity is induced by pulling-back via $\Xi$.}
\end{enumerate}
A homological $P=W$ package $(X_P,X_W,h,A,\Phi)$ is the same as a $P=W$ package except that the diffeomorphism $\Xi$ is replaced by an isomorphism $\Phi:H^*(X_W)\xrightarrow{\sim}H^*(X_P)$.
\end{definition}

\begin{theorem}\label{4.22}
If $S_P$ and $S_W$ are smooth surfaces and $(S_P,S_W,h,\A^1,\Phi)$ is a homological $P=W$ package. Then the Cartesian product $(S_P^n,S_W^n,h^n,\A^n,\Phi^n)$, the symmetric product $(S_P^{(n)},S_W^{(n)},h^{(n)},\A^n,\Phi^{(n)})$ and the Hilbert scheme $(S_P^{[n]},S_W^{[n]},h^{[n]},\A^n, \Phi^{[n]})$ are also homological $P=W$ packages, where $h^{[n]}:X_P^{[n]}\to X_P^{(n)}\to \A^n$ and
\begin{equation*}
\begin{array}{rl}
\Phi^{[n]}: H^*\left(X_W^{[n]};\Q\right)=&\displaystyle\bigoplus_{\nu=1^{a_1}\cdots n^{a_n}}\bigotimes_{i=1}^n H^*\left(X_W^{(a_i)};\Q\right)[2n-2l(\nu)]\\
\xrightarrow{\oplus \Phi^{(\nu)}} &\displaystyle\bigoplus_{\nu=1^{a_1}\cdots n^{a_n}} \bigotimes_{i=1}^n H^*\left(X_P^{(a_i)};\Q\right)[2n-2l(\nu)]=H^*\left(X_P^{[n]};\Q\right)
\end{array}
\end{equation*}
\end{theorem}

\begin{proof}
The proof is obtained by comparing the functoriality of the weight filtration for the mixed Hodge structure and the one for the perverse filtration. \\
\noindent
Step 1.
On one hand, by the K\"unneth formula for mixed Hodge structures, 
$$
W_wH^*(S_W^n;\Q)=\text{Span }\{\alpha_1\otimes\cdots\otimes\alpha_n\mid \mathfrak{w}(\alpha_1)+\cdots+\mathfrak{w}(\alpha_n)\le w\},
$$
where $\alpha_i\in H^*(S_W)$, and the function $\mathfrak{w}$ denotes the weight of a cohomology class. On the other hand, by Corollary \ref{cartesian}, 
$$
P_pH^*(S_P^n;\Q)=\text{Span }\{\alpha_1'\otimes\cdots\otimes\alpha_n'\mid \mathfrak{p}(\alpha_1')+\cdots+\mathfrak{p}(\alpha_n')\le p\},
$$
where $\alpha_i'\in H^*(S_P)$. By the hypothesis that $(S_P,S_W,h,\A^1,\Phi)$ is a homological $P=W$ package, we have that $2\mathfrak{p}(\Phi\alpha_i)=\mathfrak{w}(\alpha_i)$. This implies that
$$
W_{2k}H^*(S_W^n;\Q)=W_{2k+1}H^*(S_W^n;\Q)=P_kH^*(S_P^n;\Q).
$$
So $(S_P^n,S_W^n,h^n,\A^n,\Phi^n)$ is a homological $P=W$ package.

\medskip
\noindent
Step 2. On one hand, the mixed Hodge structure is functorial with respect to finite group quotient. So we have
$$
W_wH^*(S_W^{(n)};\Q)=\left(W_wH^*(S_W^n;\Q)\right)^{\mathfrak{S}_n}.
$$
On the other hand, by Proposition \ref{5.6}, we have
$$
P_pH^*(S_P^{(n)};\Q)=\left(P_pH^*(S_P^n;\Q)\right)^{\mathfrak{S}_n}.
$$
Then result in step 1 immediately implies 
$$
W_{2k}H^*(S_W^{(n)};\Q)=W_{2k+1}H^*(S_W^{(n)};\Q)=P_kH^*(S_P^{(n)};\Q).
$$
So $(S_P^{(n)},S_W^{(n)},h^{(n)},\A^n,\Phi^{(n)})$ is a homological $P=W$ package.

\medskip
\noindent
Step 3. On one hand, Theorem 5.3.1 in \cite{hilb} asserts that
$$
H^*\left(S_W^{[n]};\Q\right)(n)\cong\bigoplus_{\nu}H^*\left(S_W^{(\nu)};\Q\right)[2l(\nu)-2n]\left(l(\nu)\right)
$$
is an isomorphism of mixed Hodge structures, so we have
$$
W_wH^*\left(S_W^{[n]};\Q\right)=\bigoplus_\nu W_{w+2l(\nu)-2n} H^*(S_W^{(\nu)};\Q)[2l(\nu)-2n].
$$
On the other hand, by Corollary \ref{perversity}, the perverse filtration for the map $h^{[n]}:S_P^{[n]}\to \A^n $ can be expressed as 
\begin{equation*}
\begin{array}{rrl}
P_pH^*\left(S_P^{[n]};\Q\right)&=&\displaystyle\bigoplus_\nu\left(P_{pl(\nu)-n} H^*(S_P^{l(\nu)};\Q)\right)^{\mathfrak{S}_\nu}[2l(\nu)-2n]\\
&=&\displaystyle\bigoplus_\nu P_{p+l(\nu)-n} H^*(S_P^{(\nu)};\Q)[2l(\nu)-2n]
\end{array}
\end{equation*}
where the perversities on the right side are taken with respect to $h^{l(\nu)}:S_P^{l(\nu)}\to \A^{l(\nu)}$ and  $h^{(\nu)}:S_P^{(\nu)}\to \A^{l(\nu)}$ respectively.
The result in step 2 implies that 
$$
W_{2k}H^*(S_W^{[n]};\Q)=W_{2k+1}H^*(S_W^{[n]};\Q)=P_kH^*(S_P^{[n]};\Q).
$$
So $(S_P^{[n]},S_W^{[n]},h^{[n]},\A^n,\Phi^{[n]})$ is a homological $P=W$ package.
\end{proof}

\section{Applications to the $P=W$ conjecture}
\noindent
In this chapter, we will consider five families of Hitchin systems and the corresponding character varieties. We will prove the multiplicativity of the perverse filtration for the Hitchin map, compute perverse numbers and prove the full version of $P=W$ for the $n=1$ case. 
\subsection{Five families of Hitchin systems}
\noindent 
We first define the five families of Hitchin systems we consider.

\begin{enumerate}
\item{Type $\widetilde{A_0}(n)$. Consider the moduli space of degree $0$ rank $n$ parabolic Higgs bundles over an elliptic curve $(E,0)$, whose Higgs field can have at worst a first order pole at $0$ and the residue of the Higgs field at $0$ is nilpotent with respect to a multi-dimension $\{n,1,0\}$ flag.}
\item{Type $\widetilde{D_4}(n)$. Consider the moduli space of degree $0$ rank $2n$ parabolic Higgs bundles over a weighted curve $(\P^1,p_1,p_2,p_3,p_4)$, whose Higgs field can have at worst a first order pole at marked points and the residues of the Higgs field are nilpotent with respect to a multi-dimension $\{2n,n,0\}$ flag for $p_1,p_2,p_3$, and a multi-dimension $\{2n,n,1,0\}$ flag for $p_4$.}
\item{Type $\widetilde{E_6}(n)$. Consider the moduli space of degree $0$ rank $3n$ parabolic Higgs bundles over a weighted curve $(\P^1,p_1,p_2,p_3)$, whose Higgs field can have at worst a first order pole at marked points and the residues of the Higgs field is nilpotent with respect to a multi-dimension $\{3n,2n,n,0\}$ flag for $p_1,p_2$ and a multi-dimension $\{3n,2n,n,1,0\}$ flag for $p_3$.}
\item{Type $\widetilde{E_7}(n)$. Consider the moduli space of degree $0$ rank $4n$ parabolic Higgs bundles over a weighted curve $(\P^1,p_1,p_2,p_3)$, whose Higgs field can have at worst a first order pole at marked points and the residues of the Higgs field are nilpotent with respect to a multi-dimension $\{4n,2n,0\}$ flag for $p_1$, a multi-dimension $\{4n,3n,2n,n,0\}$ flag for $p_2$ and a multi-dimension $\{4n,3n,2n,n,1,0\}$ flag for $p_3$.}
\item{Type $\widetilde{E_8}(n)$. Consider the moduli space of degree $0$ rank $6n$ parabolic Higgs bundles over a weighted curve $(\P^1,p_1,p_2,p_3)$, whose Higgs field can have at worst a first order pole at marked points and the residues of the Higgs field are nilpotent with respect to a multi-dimension $\{6n,3n,0\}$ flag for $p_1$, a multi-dimension $\{6n,4n,2n,0\}$ flag for $p_2$ and a multi-dimension $\{6n,5n,4n,3n,2n,n,1,0\}$ flag for $p_3$.}
\end{enumerate}

\noindent
The geometry of the above moduli spaces of parabolic Higgs bundles are described explicitly by the following theorem  in \cite{grochenig} due to Gr\"ochenig. 

\begin{theorem}[\cite{grochenig} Theorem 4.1]
We consider the moduli of parabolic Higgs bundles in $n=1$ case for any of the five families. Let $\Gamma:=\{0\},\Z/2\Z,\Z/3\Z,\Z/4\Z,\Z/6\Z$ respectively. Let $M_D$ denote the moduli of parabolic Higgs bundles. Then $M_D$ is isomorphic to $\Gamma$-equivariant Hilbert scheme on $T^*E$, which is the crepant resolution of the quotient $T^*E/\Gamma$.
\end{theorem}

\begin{theorem}[\cite{grochenig} Theorem 5.1]\label{moduli}
Let $M_D(n)$ denote the moduli space of parabolic Higgs bundle in any of the five families, and $M_D(1)$ is abbreviated to $M_D$. Then we have
$$
M_D^{[n]}\cong M_D(n)
$$
The Hitchin map  $M_D^{[n]}\to\mathbb{A}^n$ factors through the Hilbert-Chow map
$$
M_D^{[n]}\to M_D^{(n)}\to (\mathbb{A}^1)^{(n)}=\mathbb{A}^n,
$$
where $M_D^{(n)}\to (\mathbb{A}^1)^{(n)}$ is induced by $M_D^n\to (\mathbb{A}^1)^n$.
\end{theorem}
\noindent

\noindent
In parabolic non-abelian Hodge theory, the moduli of parabolic Higgs bundle is canonically diffeomorphic to the corresponding character variety. The $P=W$ Conjecture \ref{1.1} asserts that under this canonical diffeomorphism, the weight filtration in mixed Hodge structure on the cohomology of character variety corresponds to the perverse filtration on the cohomology of the Higgs moduli space with respect to the Hitchin map. By the Simpson's table on page 720 in \cite{harmonic}, we may find the charcter varieties corresponding to our five families of moduli of parabolic Higgs bundles. The Conjecture 1.2.1 in \cite{hodge number} predicts that all cohomology class are of Hodge-Tate type and the mixed Hodge numbers depend on the multiplicities of eigenvalues of the monodromy action around the punctures rather than the eigenvalues themself. So for our purpose, we list the corresponding moduli description of character varieties for our five families of Hitchin systems without mentioning the specific eigenvalues for the monodromy action.
\begin{enumerate}
\item{Let $E$ be any elliptic curve. Consider $GL(n,\C)$-representations of $\pi_1(E\setminus p)$ such that the image of small loops around punctures are in a prescribed conjugacy class whose multiplicities of eigenvalue are of type $$(n-1,1).$$}
\item{Let $C=\P^1\setminus\{p_1,\cdots,p_4\}$. Consider $GL(2n,\C)$-representations of $\pi_1(C)$ such that the image of small loops around punctures are in a prescribed conjugacy classes whose multiplicities of eigenvalue are of type $$(n,n)(n,n)(n,n)(n,n-1,1).$$
}
\item{Let $C=\P^1\setminus\{p_1,p_2,p_3\}$. Consider $GL(3n,\C)$-representations of $\pi_1(C)$ such that the image of small loops around punctures are in a prescribed conjugacy classes whose multiplicities of eigenvalue are of type $$(n,n,n)(n,n,n)(n,n,n-1,1).$$
}
\item{Let $C=\P^1\setminus\{p_1,p_2,p_3\}$. Consider $GL(4n,\C)$-representations of $\pi_1(C)$ such that the image of small loops around punctures are in a prescribed conjugacy classes whose multiplicities of eigenvalue are of type $$(2n,2n)(n,n,n,n)(n,n,n,n-1,1).$$
}
\item{Let $C=\P^1\setminus\{p_1,p_2,p_3\}$. Consider $GL(6n,\C)$-representations of $\pi_1(C)$ such that the image of small loops around punctures are in a prescribed conjugacy classes whose multiplicities of eigenvalue are of type $$(3n,3n)(2n,2n,2n)(n,n,n,n,n,n-1,1).$$
}
\end{enumerate}

\noindent
We have the following explicit description for these character varieties for $n=1$ cases.

\begin{theorem}[\cite{character} Theorem 6.14 and 6.19]\label{10.1}
The character varieties $M_B(1)$ above can be described explicitly as follows.
\begin{enumerate}
\item{Type $\widetilde{A_0}$. $\C^*\times\C^*$.}
\item{Type $\widetilde{D_4}$. Degree $3$ del Pezzo surface with a triangle removed.}
\item{Type $\widetilde{E_6}$. Degree $3$ del Pezzo surface with a nodal $\P^1$ removed.}
\item{Type $\widetilde{E_7}$. Degree $2$ del Pezzo surface with a nodal $\P^1$ removed.}
\item{Type $\widetilde{E_8}$. Degree $1$ del Pezzo surface with a nodal $\P^1$ removed.}
\end{enumerate}
Furthermore, these del Pezzo surfaces can be expressed by an explicit formula in weighted projective space away from the singularities, and the removed triangle or nodal $\P^1$ are cut out by a hyperplane section.
\end{theorem}

\noindent
Contrary to the moduli of parabolic Higgs bundle case, we don't know much about character varieties for $n>1$. Nevertheless, there are conjectures in \cite{hodge number} which predict the behavior of the mixed Hodge numbers of character varieties. We will go back to this point in section $5.4$.

\subsection{Multiplicativity of perverse filtrations for Hitchin systems}
We will use the technique we developed in previous chapters to prove the multiplicativity of the five families of Hitchin systems.
\begin{proposition}\label{decomposition 5 cases}
Let $h:M_D\to\C$ be $n=1$ cases for the five families. Then $M_D$ is smooth and has trivial canonical bundle. The dual graph of irreducible components of the fiber over $0$ is affine Dynkin diagram $\widetilde{A_0}$, $\widetilde{D_4}$, $\widetilde{E_6}$, $\widetilde{E_7}$, $\widetilde{E_8}$, respectively. Let $\hat{h}: h^{-1}\C^*\to\C^*$ be the smooth part of the map, let $j:\C^*\to\C$, let $\hat{R}^1=R^1\hat{h}_*\Q$. Then a perverse decomposition of $h:M_D\to \C$ can be written as follows.
$$
Rh_*\Q_{M_D}[1]\cong \{\Q_\C[1]\}\bigoplus\left\{j_*\hat{R}^1\oplus\Q_0^k\right\}[-1]\bigoplus \{\Q_\C[1]\}[-2]
$$
where 
$$k=
\begin{cases}
0& \widetilde{A_0}\text{ case}\\
4& \widetilde{D_4}\text{ case}\\
6& \widetilde{E_6}\text{ case}\\
7& \widetilde{E_7}\text{ case}\\
8& \widetilde{E_8}\text{ case}.\\
\end{cases}
$$
In particular, the dimension of the perverse filtration is given by 
$$
\dim \Gr_pH^d(M_D,\Q)=
\begin{cases}
1&p=d=0\\
1&p=d=2\\
2&p=d=1,\widetilde{A_0}\text{ case}\\
k&p=1,d=2,\widetilde{D_4},\widetilde{E_6},\widetilde{E_7},\widetilde{E_8}\text{ case}\\
0&\text{otherwise.}
\end{cases}
$$
\end{proposition}

\begin{proof}
The quotient of $T^*E$ by $\Gamma$ is computed using elementary methods. We list the type of singularities in our five cases.\\
\begin{center}
\begin{tabular}{|c|c|}
\hline Case & Singularities\\
\hline $\widetilde{A_0}$& none\\
\hline $\widetilde{D_4}$& $4$ $A_1$\\
\hline $\widetilde{E_6}$& $3$ $A_2$\\
\hline $\widetilde{E_7}$& $1$ $A_1$, $2$ $A_3$\\
\hline $\widetilde{E_8}$& $1$ $A_1$, 1 $A_2$, 1 $A_5$\\
\hline
\end{tabular}
\end{center}
Note that all singularities take place in the fiber over $0$, so the dual graph of irreducible components match the affine Dynkin diagram. The action of $\Gamma$ on $T^*E$ preserves the canonical form, so the trivial canonical bundle descends to the quotient. The minimal resolution of type $A$ singularities is crepant, so $M_D$ has trivial canonical bundle. The perverse decomposition is again due to Theorem 3.2.2 of \cite{surface}. Here the map $h:M_D\to\C$ has connected fibers, so $j_*\hat{R}^0=j_*\hat{R}^2=\Q_\C$. The dimension of the perverse filtration will follow if we show $\mathbb{H}^*(j_*\hat{R}^1)=0$ in $\widetilde{D_4},\widetilde{E_6},\widetilde{E_7},\widetilde{E_8}$ cases. In fact,  the local systems $\hat{R}^1$ can be described explicitly. They are rank $2$ representation of $\Z=\pi_1(C^*)$ with monodromy 
$
\left(
\begin{array}{cc}
-1&0\\
0&-1
\end{array}
\right)
$
,
$
\left(
\begin{array}{cc}
0&-1\\
1&-1
\end{array}
\right)
$
,
$
\left(
\begin{array}{cc}
0&-1\\
1&0
\end{array}
\right)
$
,
$
\left(
\begin{array}{cc}
0&-1\\
1&1
\end{array}
\right)
$.
A simple \v{C}ech cohomology argument shows that all cohomology group of $\hat{R}^1$ vanishes, and a spectral sequence argument shows that $j_*\hat{R}^1$ also vanishes.
\end{proof}

\noindent
Although $M_D$ is non-compact, the small diagonal embedding $\Delta_{n,*}:M_D\to M_D^n$ is still proper. So we have the push-forward in Borel-Moore homology. We may still define
$$
\Delta_{*,n}:H^*(M_D)\cong H_{4-*}^{BM}(M_D)\to H_{4-*}^{BM}(M_D^n)\cong H^{*+4(n-1)}(M_D^n)
$$
The following proposition is a counterpart of Proposition \ref{diagonal}. 

\begin{proposition}\label{diagonal open}
Let $f:M_D\to\C$ be as in Proposition \ref{decomposition 5 cases}. In $\widetilde{A_0}$ case, the Gysin push-forward by the small diagonal embedding 
$\Delta_{n,*}(\gamma)=0$ for any $\gamma\in H^*(M_D)$ and $n>1$.
In the other four cases, let $E_i$ be exceptional divisors of the resolution, then 
$$
\Delta_{2,*}(1)=-\sum_{i=1}^{k}{[E_i]\otimes [E_i]}
$$
and $\Delta_{n,*}(\gamma)=0$ for any $n>2$ or $n=2$, $\gamma\neq1$.
In particular, the perversity estimation of diagonal in Proposition \ref{diagonal}
$$\mathfrak{p}(\Delta_{n,*}(\gamma))\le\mathfrak{p}(\gamma)+2(n-1)$$ is still true. 
\end{proposition}

\begin{proof}
Note that $\Delta_{n,*}$ increases the degree by $4(n-1)$, however in our cases, the top nontrivial degree for $H^*(M_D^n$) is $2n$. So when $n\ge3$, the push-forward is automatically $0$. When $n=2$, the only possible nonzero term is $\Delta_{2,*}(1)$. In the $\widetilde{A_0}$ case, $H^4(M_D\times M_D)$ is one dimensional, generated by the class $[\C]\otimes[\C]$ and $H_4(M_D\times M_D)$ is generated by $E\otimes E$. $\langle\Delta_{2,*}(1),E\otimes E\rangle_{M_D\times M_D}=\langle E,E\rangle_{M_D}=0$, so we have $\Delta_{2,*}(1)=0$. For other four cases, according to the decomposition, we pick a basis $[E_1],\cdots,[E_k],\Sigma\in H^2(M_D)$, where $\Sigma$ is a generic section of map $f:M_D\to\C$ whose perversity is $2$. To write $\Delta_{2,*}(1)$ in terms of the basis, it suffices to intersect it with the dual basis. The dual basis in $H_2(M_D)$ is $\{E_1,\cdots,E_k,F\}$, where $F$ denote the cycle class of general fiber. Since 
$$\langle\Delta_{2,*}(1),E_i\otimes E_j\rangle_{M_D\times M_D}=\langle E_i,E_j\rangle_{M_D}=-\delta_{ij},$$
$$\langle\Delta_{2,*}(1),E_i\otimes F\rangle_{M_D\times M_D}=\langle E_i,F\rangle_{M_D}=0,$$
$$\langle\Delta_{2,*}(1),F\otimes F\rangle_{M_D\times M_D}=\langle F,F\rangle_{M_D}=0.$$
We conclude that
$$
\Delta_{2,*}(1)=-\sum_{i=1}^{k}{[E_i]\otimes [E_i]}
$$ 
\end{proof}

\begin{theorem}\label{multiplicative open}
Let $f:M_D\to\C$ be as in Proposition \ref{decomposition 5 cases}. Then the perverse filtration on $H^*(M_D^{[n]})$ defined by the map $h:M_D^{[n]}\to \C^{n}$ is multiplicative .
\end{theorem}

\begin{proof}

Since $M_D=\widetilde{T^*E/\Gamma}$, we set $\overline{M_D}=\widetilde{E\times \P^1/\Gamma}$. By the explicit geometry we know that the restriction map $H^*(\overline{M_D})\to H^*(M_D)$ is surjective. So we can use Proposition \ref{4.10} to compute the cup product. The method we use is similar to the proof of Theorem \ref{multiplicative^n}. We fix filtered basis $B'$ as follows.

\begin{enumerate}
\item{Type $\widetilde{A_0}$. Let $B'=\{1,\alpha,\beta, \alpha\cup\beta\}$, where $\alpha,\beta$ are basis of $H^1(M_D)$.}
\item{Type $\widetilde{D_4},\widetilde{E_6},\widetilde{E_7},\widetilde{E_8}$. Let $B'=\{1,E_1,\cdots,E_k,\Sigma\}$.}
\end{enumerate}

\noindent
By Proposition \ref{decomposition 5 cases} and Proposition \ref{diagonal open}, the basis $B'$ is filtered with respect to the perverse filtration for the map $h:M_D\to\C$ and has the perverse estimation of diagonal embedding. So the hypotheses of Lemma \ref{pull-back} and Lemma \ref{push-forward} are satisfied. Now by Proposition \ref{4.15}, it suffices to prove that the abstract perverse filtration defined on $H^*(M_D)\{\mathfrak{S}_n\}$ is multiplicative. Furthermore, it suffices to prove the result for our filtered basis, namely
\begin{equation*}
\begin{array}{rl}
 &\displaystyle \mathfrak{p}\left(\bigotimes_{i=1}^n\bigotimes_{j=1}^{a_i}\alpha_{ij}\cdotp\sigma\cup\bigotimes_{i=1}^n\bigotimes_{j=1}^{a_i'}\alpha'_{ij}\cdotp\tau\right)\\
\le&\displaystyle 
\sum_{i=1}^n{\sum_{j=1}^{a_i}{\mathfrak{p}(\alpha_{ij})}}+\sum_{i=1}^n{(i-1)a_i}\\
+&\displaystyle 
\sum_{i=1}^n{\sum_{j=1}^{a'_i}{\mathfrak{p}(\alpha'_{ij})}}+\sum_{i=1}^n{(i-1)a'_i},
\end{array}
\end{equation*}
where $\alpha_{ij}$ and $\alpha'_{ij}$ run over basis $B'$.  Note that the cup product formula computes independently on each orbit of $\langle\sigma,\tau\rangle$-action on $[n]$ individually. Let $O$ be an orbit of the action $\langle\sigma,\tau\rangle$ on $[n]$, i.e. $|\langle\sigma,\tau\rangle\backslash O|=1$. The product is computed by
\begin{equation*}
\begin{array}{rrl}
A^{\otimes \langle\sigma\rangle\backslash O}\cdotp\sigma|_O\otimes A^{\otimes \langle\tau\rangle\backslash O}\cdotp\tau|_O &\to&
A^{\otimes \langle\sigma\tau\rangle\backslash O}\cdotp\sigma\tau|_O\\
a\cdotp\sigma|_O\otimes a'\cdotp\tau|_O&\mapsto& f_{\langle\sigma,\tau\rangle,\langle\sigma\tau\rangle}(f^{\langle\sigma\rangle,\langle\sigma,\tau\rangle}(a)\cdotp
f^{\langle\tau\rangle,\langle\sigma,\tau\rangle}(a')\cdotp e^{g(\sigma,\tau)})\cdotp\sigma\tau|_O
\end{array}
\end{equation*}
for every $O$. Note that the Euler class $e$ is of top degree and $M_D$ is smooth and non-compact, so that it suffices to consider the case when $g(\sigma,\tau)=0$. By Lemma \ref{pull-back} and Lemma \ref{push-forward}, we have
\begin{equation*}
\begin{array}{rrl}
& &\displaystyle\mathfrak{p}\left(f_{\langle\sigma,\tau\rangle,\langle\sigma\tau\rangle}(f^{\langle\sigma\rangle,\langle\sigma,\tau\rangle}(a)\cdotp
f^{\langle\tau\rangle,\langle\sigma,\tau\rangle}(a')\cdotp e^{g(\sigma,\tau)})\cdotp\sigma\tau|_O\right)\\
&=&\displaystyle\mathfrak{p}\left(f_{\langle\sigma,\tau\rangle,\langle\sigma\tau\rangle}(f^{\langle\sigma\rangle,\langle\sigma,\tau\rangle}(a)\cdotp
f^{\langle\tau\rangle,\langle\sigma,\tau\rangle}(a'))\right)+|O|-|\langle\sigma\tau\rangle\backslash O|\\
&=&\displaystyle\mathfrak{p}\left(f^{\langle\sigma\rangle,\langle\sigma,\tau\rangle}(a)\cdotp
f^{\langle\tau\rangle,\langle\sigma,\tau\rangle}(a'))\right)+2(|\langle\sigma\tau\rangle\backslash O|-1)+|O|-|\langle\sigma\tau\rangle\backslash O|\\
&=&\displaystyle\mathfrak{p}(a)+\mathfrak{p}(a')+|O|+|\langle\sigma\tau\rangle\backslash O|-2\\
&=&\displaystyle\mathfrak{p}(a\cdotp\sigma)-(|O|-|\langle\sigma\rangle\backslash O|)+\mathfrak{p}(a'\cdotp\tau)-(|O|-|\langle\tau\rangle\backslash O|)+|O|+|\langle\sigma\tau\rangle\backslash O|-2\\
&=&\displaystyle\mathfrak{p}(a\cdotp\sigma)+\mathfrak{p}(a'\cdotp\tau)-2g(\sigma,\tau)\\
&=&\displaystyle\mathfrak{p}(a\cdotp\sigma)+\mathfrak{p}(a'\cdotp\tau)
\end{array}
\end{equation*}
\end{proof}

\noindent
Combining Theorem \ref{multiplicative open} and Theorem \ref{moduli}, we have
\begin{theorem}
For the five families of moduli space of parabolic Higgs bundles described in Theorem \ref{moduli}, the perverse filtration defined by the Hitchin map is multiplicative.
\end{theorem}

\subsection{Full version of $P=W$ for $n=1$}
\noindent
In this section, we prove the full version of $P=W$ conjecture in the $n=1$ cases by using the explicit geometry of the Hitchin map.

\medskip
\noindent
I thank Dingxin Zhang for suggesting the following lemma to me.

\begin{lemma}\label{10.2}
Let $X$ be any of the del Pezzo surface in Theorem \ref{10.1}. Let $i:T\to X$ be the closed embedding of the removed curve in Theorem \ref{10.1} and let $j:U\hookrightarrow X$ be its complement. Then 
$$
W_2H^2(U)\cong \im\left(H^2_c(U)\to H^2(U)\right).
$$ 
\end{lemma}

\begin{proof}
We have a diagram where the row and the column are distinguished triangles

$$
\begin{tikzcd}
\ &i_*i^!\Q_X\ar{d}&\ \\
Rj_!\Q_U\ar{r}&\Q_X\ar{r}\ar{d}& i_*\Q_T\\
\ & Rj_*\Q_U&\ \\
\end{tikzcd}
$$

Taking cohomology in degree $2$, we have
$$
\begin{tikzcd}
\ &H^2_T(X)\ar{d}{i_*}\ar{rd}{\psi}&\ \\
H^2_c(U)\ar{r}{j_!}\ar{rd}[swap]{\phi}& H^2(X)\ar{r}{i^*}\ar{d}{j^*}& H^2(T)\\
\ &H^2(U)&\ 
\end{tikzcd}
$$
\noindent
By \cite{deligne}, Corollaire 3.2.17, the image of $j^*$ is precisely $W_2H^2(U)$. So it suffices to prove that $\im j_!+\ker j^*=H^2(X)$. By exactness, this is equivalent to proving that $\im i_*+\ker i^*=H^2(X)$. Therefore, it suffices to prove that $\psi=i^*i_*$ is an isomorphism. In fact, the morphism $\psi$ maps $\xi\in  H^2_T(X)\cong H_2(T)$ to $\xi^\dagger:H_2(T)\to \Q$, where $\xi^\dagger(\gamma)=\int_X\xi\cup i_*(\gamma)$. This defines a symmetric bilinear form on $H_2(T)$ defined by the intersection number of components of $T$ viewed in $X$. To show that $\psi$ is an isomorphism, it suffices to show that this bilinear form is nondegenerate. In the case where $T$ is a triangle, then the intersection matrix is
$$
\left(
\begin{array}{ccc}
-1&1&1\\
1&-1&1\\
1&1&-1
\end{array}
\right)
$$
In the case when $T$ is a nodal $\P^1$, since it is cut out by hyperplane section away from the singularities of the weighted projective space, so it is an ample divisor, therefore the self intersection of $T$ is nonzero.
\end{proof}

\begin{theorem}
The perverse filtration for the map $M_D(1)\to \C$ and the mixed Hodge filtration on $M_B(1)$ correspond.
\end{theorem}

\begin{proof}
By the spectral sequence of weight filtration of mixed Hodge structure, dimensions of graded pieces of weight filtration is easily computed. 
$$
\dim \Gr_w^WH^d(M_B)=
\begin{cases}
1 & w=d=0\\
1 & w=4,d=2\\
2 & w=2,d=1, \widetilde{ A_0}\text{ case}\\
k & w=2,d=2, \widetilde{ D_4},\widetilde{E_6},\widetilde{E_7},\widetilde{E_8}\text{ case}\\
0& \text{otherwise}.
\end{cases}
$$  
Compare with Proposition \ref{decomposition 5 cases}, numerical $P=W$ holds for $n=1$ in our five cases. To prove the full version of the $P=W$ conjecture, it suffices to prove that $P_1H^2(M_D)=W_2H^2(M_B)$ in $\widetilde{D_4},\widetilde{E_6},\widetilde{E_7},\widetilde{E_8}$ cases. ($\widetilde{A_0}$ cases is trivially true.) By Proposition \ref{decomposition 5 cases}, $P_1H^2(M_D)$ in four cases are all spanned by the fundamental classes of exceptional curves of the minimal resolutions, and $H^2_c(M_D)$ is generated by the fundamental classes of exceptional curves and a generic fiber of the morphism $M_D\to \C$. Note that the fiber class is $0$ in $H^2(M_D)$, so we have 
$$
P_1H^2(M_D)=\im\left(H^2_c(M_D)\to H^2(M_D)\right).
$$
By Lemma \ref{10.2}, we have $W_2H^2(M_B)=\im(H^2_c(M_B)\to H^2(M_B))$. This completes the proof.
\end{proof}

\begin{remark}
In \cite{P=W}, Theorem 3.2.1 asserts that the forgetful map $H_c^{6g-6}(M_D)\to H^{6g-6}(M_D)$ is the zero map, where $6g-6$ is the complex dimension of the moduli space in the context. In fact, they consider moduli space of degree $1$ Higgs bundles and twisted representations, so their result does not contradict ours.  
\end{remark}

\subsection{Perverse numbers and numerical $P=W$}
\noindent
In this section, we give some partial numerical evidence for $P=W$ Conjecture \ref{1.1} for our five families of Hitchin fibrations. We use Proposition \ref{4.4} and Proposition \ref{decomposition 5 cases} to compute the perverse numbers $\Gr_p H^d(M_D^{[n]})$. Conjecture 1.2.1 in \cite{hodge number} which predicts that the mixed Hodge numbers of character varieties can be computed by a combinatorial formula. We made a conjecture that in our five families of Hitchin systems, the perverse numbers equal the conjectural mix Hodge numbers of the corresponding character varieties. We have verified our conjecture for small $n$.

\begin{theorem}\label{8.1}
Let $f:M_D\to\C$ be $n=1$ case of the five families. Denote perverse numbers by $p^{i,j}=\dim \Gr_iH^j(M_D^{[n]})$. Let perverse Poincar\'e polynomial be $P_n(q,t)=\sum_{i,j}p^{i,j}q^it^j$. Then for the $\widetilde{A_0}$ case, the generating series is
$$
\sum_{n=0}^\infty s^nP_n(q,t)=\prod_{m=1}^\infty\frac{(1+s^mq^{m}t^{2m-1})^2}{(1-s^mq^{m-1}t^{2m-2})(1-s^mq^{m+1}t^{2m})}.
$$
For the other four cases, the generating series are
$$
\sum_{n=0}^\infty s^nP_n(q,t)=
\prod_{m=1}^\infty\frac{1}{(1-s^mq^{m-1}t^{2m-2})(1-s^mq^mt^{2m})^k(1-s^mq^{m+1}t^{2m})}
$$
where $k$ is defined in Proposition \ref{decomposition 5 cases}.
\end{theorem} 

\begin{proof}
We prove the equalities by expanding both hand sides and identifying the corresponding terms. Since all cases are similar, we prove the $\widetilde{D_4}$ case as an illustration of the calculations. By Proposition \ref{4.4}, K\"unneth formula and MacDonald theorem, we have 

$$
\begin{array}{rl}
 &H^*\left(M_D^{[n]}\right)[n]\\
= &\displaystyle\mathbb{H}\left(\bigoplus_{\nu=1^{a_1}\cdots n^{a_n}} Rr_{C,*}^{(\nu)}\left(\bigoplus_{\bold{r}+\bold{s}+\bold{t}=\bold{a}}\mathcal{P}_0^{\{\bold{r}\}}\boxtimes \mathcal{P}_1^{(\bold{s})}\boxtimes \mathcal{P}_2^{\{\bold{t}\}}\right)[-C(\bold{r},\bold{s},\bold{t})]\right)\\
= &\displaystyle\bigoplus_{\nu=1^{a_1}\cdots n^{a_n}}\bigoplus_{\bold{r}+\bold{s}+\bold{t}=\bold{a}}\mathbb{H}\left(\mathcal{P}_0^{\{\bold{r}\}}\boxtimes \mathcal{P}_1^{(\bold{s})}\boxtimes \mathcal{P}_2^{\{\bold{t}\}}\right)[-C(\bold{r},\bold{s},\bold{t})]\\
= &\displaystyle\bigoplus_{\nu=1^{a_1}\cdots n^{a_n}}\bigoplus_{\bold{r}+\bold{s}+\bold{t}=\bold{a}}\mathbb{H}\left(\mathcal{P}_0^{\{\bold{r}\}}\right)\otimes \mathbb{H}\left(\mathcal{P}_1^{(\bold{s})}\right)\otimes \mathbb{H}\left(\mathcal{P}_2^{\{\bold{t}\}}\right)[-C(\bold{r},\bold{s},\bold{t})]\\
\end{array} 
$$

\noindent
By Proposition \ref{decomposition 5 cases}, we have $\mathcal{P}_0=\Q_\C[1]$, $\mathcal{P}_1=j_*\hat{R}^1[1]\oplus\Q_0^4$, $\mathcal{P}_2=\Q_\C[1]$. Note that $j_*\hat{R}^1$ has no cohomology so that we have

\begin{center}
\begin{tabular}{|c|c|c|c|}
\hline   & $\mathbb{H}(\mathcal{P}_0)$ & $\mathbb{H}(\mathcal{P}_1)$ & $\mathbb{H}(\mathcal{P}_2)$\\
\hline dimension & $1$ & $4$ & $1$\\
\hline degree & $-1$ & $0$ & $-1$\\
\hline
\end{tabular}
\end{center}

\noindent
Together with Proposition \ref{product} and Remark \ref{quotient}, the table shows that the summand
$$
\mathbb{H}\left(\mathcal{P}_0^{\{\bold{r}\}}\right)\otimes \mathbb{H}\left(\mathcal{P}_1^{(\bold{s})}\right)\otimes \mathbb{H}\left(\mathcal{P}_2^{\{\bold{t}\}}\right)[-C(\bold{r},\bold{s},\bold{t})]
$$
is a vector space of dimension $\displaystyle\prod_{i=1}^n\left(\begin{array}{c}s_i+3\\3\end{array}\right)$, and all cohomology classes in this summand are of cohomological degree $\sum_{i=1}^n-r_i+s_i+t_i$ and perversity $\sum_{i=1}^ns_i+2t_i$. Here we use the fact that $\mathbb{H}(\mathcal{P}_0)$ and $\mathbb{H}(\mathcal{P}_2)$ are in odd degree, so that 

$$
\mathbb{H}\left(\mathcal{P}_0^{\{\bold{r}\}}\right)=\bigotimes_{i=1}^n Sym^{r_i} \mathbb{H}^{-1}(\mathcal{P}_0)=\C,
$$
and similarly for $\mathbb{H}\left(\mathcal{P}_2^{\{\bold{t}\}}\right)$. So 
$$
\begin{array}{rl}
P_n(q,t)=&\displaystyle\sum_{\nu=1^{a_1}\cdots n^{a_n}}q^{n-l(\nu)}t^{2n-l(\nu)}\prod_{i=1}^n\sum_{r_i+s_i+t_i=a_i}\left(\begin{array}{c}s_i+3\\3\end{array}\right)q^{s_i+2t_i}t^{-r_i+s_i+t_i}\\
=&\displaystyle\sum_{\nu=1^{a_1}\cdots n^{a_n}}q^{n-l(\nu)}t^{2n-l(\nu)}\prod_{i=1}^n t^{-a_i}\sum_{r_i+s_i+t_i=a_i}\left(\begin{array}{c}s_i+3\\3\end{array}\right)q^{s_i+2t_i}t^{2s_i+2t_i}\\
=&\displaystyle\sum_{\nu=1^{a_1}\cdots n^{a_n}}q^{n-l(\nu)}t^{2n-2l(\nu)}\prod_{i=1}^n \sum_{r_i+s_i+t_i=a_i}\left(\begin{array}{c}s_i+3\\3\end{array}\right)q^{s_i+2t_i}t^{2s_i+2t_i}\\
\end{array}
$$
Therefore 
$$
\sum_{n=0}^\infty P_n(q,t)=\sum_{n=0}^\infty\sum_{\nu=1^{a_1}\cdots n^{a_n}}q^{n-l(\nu)}t^{2n-2l(\nu)}\prod_{i=1}^n \sum_{r_i+s_i+t_i=a_i}\left(\begin{array}{c}s_i+3\\3\end{array}\right)q^{s_i+2t_i}t^{2s_i+2t_i}
$$

\noindent
On the other hand, to get a term in the product generating series with factor $s^n$ is equivalent to the following data: (1) A partition $\nu=1^{a_1}\cdots n^{a_n}$ of $n$, such that the factor $m=i$ contributes $(s^{i})^{a_i}$, (2) a triple $(r_i,s_i,t_i)$ for each $i$ satisfying $r_i+s_i+t_i=a_i$, such that the expansions of three parenthesis contribute $(s^i)^{r_i}$, $(s^i)^{s_i}$ and $(s^i)^{t_i}$, respectively. So the term obtained in this way is 
$$
\begin{array}{rl}
 &\displaystyle\prod_{i=1}^n\left(\begin{array}{c}s_i+3\\3\end{array}\right) s^{i(r_i+s_i+t_i)}q^{(i-1)r_i+is_i+(i+1)t_i}t^{(2i-2)r_i+2is_i+2it_i}\\
=&\displaystyle\prod_{i=1}^n\left(\begin{array}{c}s_i+3\\3\end{array}\right) s^{ia_i}q^{(i-1)a_i+s_i+2t_i}t^{(2i-2)a_i+2s_i+2t_i}\\
=&\displaystyle s^nq^{n-l(\nu)}t^{2n-2l(\nu)}\prod_{i=1}^n\left(\begin{array}{c}s_i+3\\3\end{array}\right)q^{s_i+2t_i}t^{2s_i+2t_i}.\\
\end{array}
$$
Here we use the fact that $a_i=r_i+s_i+t_i$, $n=\sum_{i=1}^n ia_i$ and $l(\nu)=\sum_{i=1}^n a_i$. By comparing with the expansion of the additive generating series, the theorem follows.
\end{proof}

\noindent
If we believe the $P=W$ Conjecture, the perverse numbers of Hitchin system should equal the mixed Hodge numbers of corresponding character varieties. In fact, the Conjecture 1.2.1 in \cite{hodge number} and Theorem \ref{8.1} suggest the following conjecture.

\begin{conjecture}
Let $\bm{\mu}$ be multi-partition which encodes the parabolic data of $M_B(n)$. Let $k$ be defined as in Proposition \ref{decomposition 5 cases}. Let $\mathbb{H}_{\bm{\mu}}$ be defined as in section 1.1 of \cite{hodge number}. Then for $\widetilde{A_0}(n)$ case, we have
$$
\prod_{m=1}^\infty\frac{(1+s^mq^{m}t^{2m-1})^2}{(1-s^mq^{m-1}t^{2m-2})(1-s^mq^{m+1}t^{2m})}=\sum_{n=0}^\infty (st^2q)^n\mathbb{H}_{\bm{\mu}}(-\sqrt{q},\frac{\sqrt{q}}{t})
$$
For the other four cases, we have
$$
\prod_{m=1}^\infty\frac{1}{(1-s^mq^{m-1}t^{2m-2})(1-s^mq^mt^{2m})^k(1-s^mq^{m+1}t^{2m})}=\sum_{n=0}^\infty (st^2q)^n\mathbb{H}_{\bm{\mu}}(-\sqrt{q},\frac{\sqrt{q}}{t})
$$
\end{conjecture} 

\begin{remark}
In fact, the conjecture for $\widetilde{A_0}$ case is the cohomological version of Conjecture 4.2.1 in \cite{hodge number 2}, and the other four cases are new. With the help of Mathematica, we prove for $n\le 6$ for $\widetilde{D_4}$ case, $n\le 4$ for $\widetilde{E_6}$ case, $n\le3$ for $\widetilde{E_7}$ and $n\le2$ for $\widetilde{E_8}$ case. Efforts to prove for general $n$ so far all ended with combinatorial difficulties. More understanding on $q,t$-Kostka numbers would be helpful.
\end{remark}

\end{document}